\documentclass[a4paper,12pt]{article}

\usepackage{pstricks}
\usepackage{graphicx,psfrag}

\usepackage{amsmath}
\usepackage{amssymb}
\usepackage{amsthm}
\usepackage{color}

\usepackage{bm}

\usepackage[dvipsnames]{xcolor}
\usepackage{hyperref}

\hypersetup{
	colorlinks=true,       % false: boxed links; true: colored links
	linkcolor=RoyalBlue,   % color of internal links (sections, etc.)
	citecolor=PineGreen,   % color of links to bibliography
	filecolor=magenta,     % color of file links
	urlcolor=Cyan,         % color of external links
	pdfpagemode=UseNone    % hides bookmarks panel on open
}

\usepackage[top=2cm,bottom=2cm,left=2cm,right=2cm]{geometry}

\DeclareMathOperator{\const}{const}
\DeclareMathOperator{\conv}{conv}
\DeclareMathOperator{\extr}{extr}

\newcommand{\eqdef}{\stackrel{\mathrm{def}}{=}}
\newcommand {\R}{{\mathbb R}}

\newcommand {\FF}{F}
\newcommand {\FFF}{\mathcal{F}}

\newcommand {\PP}{P}
\newcommand {\PPP}{\mathcal{P}}

\newcommand {\KKK}{M}
\newcommand {\SSS}{\mathcal{S}^1}

\newcommand {\nneg}{\bm{\neg}}

  \newtheorem{theorem}{Theorem}
  \newtheorem{lemma}{Lemma}

\theoremstyle{remark}
  \newtheorem{remark}{Remark}

\theoremstyle{definition}
  \newtheorem{definition}{Definition}
  
  \newtheorem{corollary}{Corollary}

\title{Isoperimetric problem for quasi-Finsler metrics and shapes of least aerodynamic resistance}

\author{Lev Lokutsievskiy\thanks{Steklov Mathematical Institute of RAS and HSE University.} \and Alexander Plakhov\thanks{Center for R{\&}D in Mathematics and Applications, Department of Mathematics, University of Aveiro, Portugal and Institute for Information Transmission Problems, Moscow, plakhov@ua.pt}}

\begin{document}
\maketitle

\begin{abstract}
We consider the following model in the framework of Newtonian aerodynamics: a 2D convex body moves forward and slowly oscillates in a rarefied medium on the plane. The law of oscillations is given. The problem is to find a body of fixed area that has the smallest resistance. We solve this problem by reducing it to an isoperimetric problem for quasi-Finsler metrics. Further, we introduce two criteria of smallness of oscillations. We show that the optimal body has singularity at the front (back) part of its boundary iff the former (latter) criterion is satisfied. The rest of the boundary is smooth. Finally, we find exact optimal shapes in the case of uniform oscillations and construct several shapes explicitly.
\end{abstract}

\begin{quote}
\end{quote}

\begin{quote}
\end{quote}

\section{Introduction}\label{sec1}

{\bf 1.1.} \ The classical isoperimetric problem seeks a domain of a given volume that minimizes its surface area. When the ambient space is equipped with a Finsler or quasi-Finsler metric, the optimal domain is not a ball, and its shape depends on the chosen metric.
Such optimization problems arise naturally in physics and the calculus of variations, with the most famous example being Newton's problem of minimal aerodynamic resistance.

Originally formulated by Isaac Newton in 1687, the aerodynamic problem seeks the optimal shape of a body moving through a rarefied medium that minimizes the drag force. Mathematically, the resistance is expressed as a surface integral of a function depending on the local normal vector. When an area or volume constraint is imposed on the body, the search for the shape of least resistance translates directly into an isoperimetric problem in a quasi-Finsler space. In this paper, we investigate this quasi-Finsler isoperimetric problem to find the planar shapes of least aerodynamic resistance for bodies undergoing complex combinations of translational and rotational motions.
\vspace{2mm}

{\bf 1.2.} \ Isaac Newton in \cite{N} considered the problem of least resistance of moving bodies in rarefied media. He assumed that a convex body performs purely translational motion, without rotation, and that the reflections of point particles of the medium from the body are perfectly elastic. The particles are originally at rest. Under these assumptions, increasing the length of the body and decreasing its width favor a decrease in resistance. Thus, by elongating the body along the direction of motion and reducing its cross section, one can make the resistance of, say, a body of fixed volume arbitrarily small. To prevent this degenerate minimum, Newton imposed a restriction on the body's length and assumed that its orthogonal cross section is a circle of fixed size.

Newton originally considered axisymmetric bodies. Starting from the early 1990s, many works have extended the least resistance problem to various classes of bodies without the symmetry assumption, as well as to classes of nonconvex bodies (see, e.g., \cite{BK,BrFK,BFK,CL1,Guasoni,Kaw,LO,LP1,DAN2003,RMS2009,SIC,SIREV,P-sing,P-boundary,W,LZ,LWZ}).

A realistic physical setting implies that a body, along with its translation, also undergoes rotational motions. In this paper, we consider a convex $2$-dimensional body on the plane that performs periodic rotational motions alongside its translation. It is assumed that rotation is very slow compared to translational motion.\footnote{More precisely, the product of the angular velocity and the size of the body at any time instant is much smaller than the translational velocity, so that particles are reflected from the body as if there were no rotation.} We are interested in minimizing the time-averaged aerodynamic resistance.

At each time instant, the body is turned by an angle $\psi$ with respect to a selected reference position, which we consider to be the position of equilibrium. We assume that the probability density distribution $f$ of the rotation angles is known. Specifically, the fraction of time the body spends rotated by an angle in the infinitesimal interval $[\psi, \psi+d\psi]$ is equal to $f(\psi) d\psi$, where
$$\int_{-\pi}^{\pi} f(\psi)\, d\psi = 1.$$

Consider a reference frame $xOy$ attached to the body $B$. In this frame, as time varies, one observes incident medium flows under different angles. We aim to calculate the component of the force (along the direction of translational motion) applied to an infinitesimal boundary segment of length $ds = \sqrt{dx^2 + dy^2}$ with the outward normal $e_\theta \eqdef (\cos\theta,\sin\theta)$ for $\theta \in [-\pi, \pi].$

Suppose that at a given moment the body is turned by the angle $\psi$ with respect to the equilibrium position. Then the angle of incidence of the flow relative to the normal at the chosen point is $\varphi = \psi + \theta$; see Fig.~\ref{fig:body}.

\begin{figure}[ht]
	\begin{center}
		\includegraphics[width=0.5\textwidth]{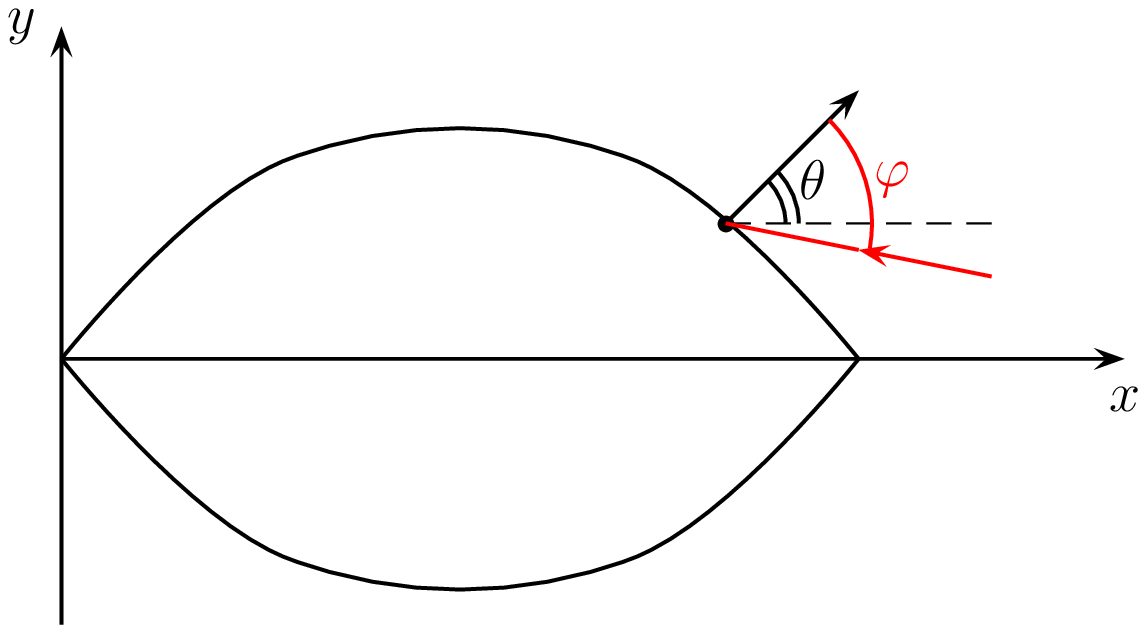}
		\caption{Oscillating body.}
		\label{fig:body}
	\end{center}
\end{figure}

The pressure of the flow on the boundary segment equals $2\rho v^2 \cdot \cos_+^3\varphi$, where $z_+ = \max \{ z, 0\}$ denotes the positive part of a real number $z$, $\rho$ is the density of the medium, and $v$ is the velocity of translation. For simplicity, we choose physical units such that $2\rho v^2 = 1$.

Since the density of the distribution characterizing the angle $\psi$ is $f$, we conclude that the time-averaged force acting on the chosen infinitesimal segment is given by
\begin{equation}\label{P}
\PPP(\theta)\, = \, \int_{-\pi}^{\pi} \cos_+^3(\psi + \theta)\, f(\psi)\, d\psi\, =\, \int_{-\pi/2}^{\pi/2} \cos^3 \varphi\, f(\varphi-\theta)\, d\varphi.
\end{equation}
\vspace{2mm}

{\bf 1.3.} \
Consider the particular case when the body performs uniform angular oscillations of amplitude $0< a \le \pi$ about its equilibrium position. By uniform, we mean that the rotational motion alternates between counterclockwise and clockwise rotation with a constant, infinitesimally small angular velocity. The extreme case $a=\pi$ corresponds to full uniform rotation with small angular velocity. Under these assumptions, the angular probability density is a uniform distribution:
$$
f(\theta) = \frac{1}{2a}\, \chi_{[-a,a]}(\theta),
$$
where $\chi$ denotes the indicator function. Consequently, the time-averaged force on a boundary segment with normal angle $\theta$ becomes
\begin{equation}\label{Punif}
\PPP(\theta)\, =\, \frac{1}{2a}\,\int_{\theta-a}^{\theta+a} \cos_+^3\varphi\, d\varphi.
\end{equation}
Note that by the symmetry of the uniform distribution, the resulting pressure function $\PPP$ is even.

\section{Statement of the problem}

{\bf 2.1.} \ Consider the Euclidean plane with the orthogonal coordinates $x,\, y$. A compact convex set in $\R^2$ is called a {\it convex body}, or just a {\it body}. Let $\PPP$ be a nonnegative continuous function on $\SSS \eqdef \mathbb{R}/2\pi\mathbb{Z}$ and $A > 0$. We state the following problem:
\begin{equation}\label{F}
	\text{Minimize the functional} \qquad \FFF(B) \ \eqdef\ \int_{\partial B} \PPP({\theta_\xi})\, d\xi
\end{equation}
in the class of convex bodies $B$ of area $A$. Here  $(\cos\theta_\xi,\sin\theta_\xi) =: n_\xi$ is the outward normal to $B$ at the regular point $\xi \in \partial B$, and $d\xi$ is the element of length. Note that the set of regular points has full measure in\footnote{The proof is completely similar to the proof of the very well known fact: any monotone functions may have only countably many jumps.} $\partial B$, and therefore, the integral in \eqref{F} is well defined.

This functional has a clear mechanical meaning: $\PPP$ is the time-averaged pressure exerted on an element of the body's boundary inclined at the corresponding angle, and $\FFF(B)$ is the time-averaged resistance of the body $B$. That is, one needs to minimize the average resistance of the body with area $A.$

Problem \eqref{F} is scale invariant: if $B$ solves the problem for bodies of area $A$, then $\lambda B$ solves the problem for bodies of area $\lambda^2 A$. Hence the problem can be stated equivalently as follows,
\begin{equation}\label{ProblEquiv}
\dfrac{\FFF(B)}{\sqrt{|B|}}\, \to\, \inf.
\end{equation}
Here and in what follows, $|B|$ means the area of $B$. With a slight abuse of notation, $|\cdot|$ will also mean the linear Lebesgue measure on $\partial B$. It will always be clear from the context which measure is meant.
\vspace{2mm}

{\bf 2.2.} \ Later on in this article we will need another expression for the functional, which is based on the notion of surface area measure. First we give several definitions.

\begin{definition}
A Borel measure $\nu$ on $S^1$ is called {\it zero-mean measure}, if it satisfies the equality
$$\int_{S^1} n\, d\nu(n) = 0.$$
\end{definition}

\begin{definition}
The surface area measure $\nu_B$ of a body $B$ is the Borel measure on $S^1$ defined by
$$
\nu_B(\Sigma) \eqdef |\{ \xi \in \partial B : n_\xi \in \Sigma \}|
$$
for any Borel set $\Sigma \subset S^1$.
\end{definition}

According to Alexandrov's Theorem, the map $B \mapsto \nu_B$ is a one-to-one correspondence between the set of convex bodies (defined up to translations) and the set of zero-mean measures (see \cite{Aleksandrov}).\footnote{Note that  in arbitrary dimension greater than or equal to 2, the surface area measure of a convex body is well defined and Alexandrov's Theorem (with some amendments) holds true.}

The least resistance problem~\eqref{F} can be equivalently expressed in terms of the surface area measure $\nu_B$ as follows. Denote
\begin{equation}\label{Fmeas}
\FF(\nu) = \int_{S^1} \PP(n)\, d\nu(n),
\end{equation}
where $\PP(e_\theta) \eqdef \PPP(\theta)$, $S^1$ is as usual the unit circumference on the plane,
and let $\nu \mapsto B_\nu$ be the map inverse to $B \mapsto \nu_B$.Then problem~\eqref{F} is as follows,
\begin{equation}\label{nu}
\text{Minimize} \quad \FF(\nu) \quad \text{in the class of zero-mean measures $\nu$ satisfying} \ \, |B_\nu| = A.
\end{equation}

%In what follows, we will use both notations $\PP,\FF$ and $\PPP,\FFF$.

\bigskip

{\bf 2.3.}
Our main goal is to find the exact optimal shape. To that end, we first solve an isoperimetric problem for a quasi-Finsler metric, which is then used to solve the problem of least resistance. Namely, we consider the problem
\begin{equation}\label{Pfinsler}
    \text{Minimize}\quad
    \int_0^1\mu_{\Omega}(\dot y,-\dot x)\,dt
    \quad\text{s.t.}\quad
    \frac12\int_0^1(x\dot y-y\dot x)\,dt=A
\end{equation}
where $\mu_\Omega$ is the Minkowski function of the set $\Omega$ given by
\begin{equation}
	\label{eq: Omega}
\Omega\ \eqdef\ \Big\{(u,v)\,\Big|\, \sqrt{u^2+v^2}\,\PPP(\operatorname{Arg}(u+iv)) \le 1 \Big\}.
\hspace{32mm}
\end{equation}
This isoperimetric problem turns out to be very useful for our context. Specifically, we prove that the optimal bodies in both problems \eqref{F} and \eqref{Pfinsler} coincide and are similar to the polar set of the convex hull of $\Omega$ where the similarity ratio is proportional to the square root of the area.
\vspace{2mm}

\textbf{2.4.}\  Let us state here the objects and notions that will be needed later on.
\vspace{2mm}

$\bullet$ \ $S^1$ is the unit circumference in $\mathbb{R}^2$, and $\SSS=\mathbb{R}/2\pi\mathbb{Z}$.
\vspace{2mm}

$\bullet$ \ $B$ is a convex body moving and oscillating on the plane, and $\nu_B$ is its surface area measure.
\vspace{2mm}

$\bullet$ \ $f : \SSS \to \mathbb{R}$ determines the type of oscillations.
\vspace{2mm}

$\bullet$ \ $e_\theta =(\cos\theta,\sin\theta)$ and $e_\theta^\perp = (-\sin\theta,\cos\theta)$.
\vspace{2mm}

$\bullet$ \ $\PPP : \SSS \to \mathbb{R}$ determines the time-averaged pressure on the body's boundary, and

$\PP(e_\theta) = \PPP(\theta).$
\vspace{2mm}

The plan of the paper is the following. In the next Section \ref{sec: convex} the isoperimetric problem is solved for the case when $\Omega$ is convex. In Section \ref{sec: non convex} the solution is obtained for the case of arbitrary (generally nonconvex) $\Omega$. Finally, in Section \ref{sec application} the obtained result is applied to the minimal resistance problem. It is shown that the optimal shape has no more than two singularities, one at the top point of the boundary, and the other on the back point, and necessary and sufficient conditions for the presence of these singularities are obtained. Additionally, the problem of uniform oscillations is solved exactly, and several optimal shapes are constructed explicitly.

\section{Solution of the problem for convex \texorpdfstring{$\Omega$}{Ω}}\label{sec: convex}

Here we assume that $\Omega$ is a convex set and provide a complete answer to the problem in this case. %In the section~\ref{sec: non convex} we consider the general non-convex case and show that the solution is easily obtained by the convexification $\Omega\longrightarrow\conv\Omega$

Let $\theta \in \SSS = \mathbb{R} / 2\pi\mathbb{Z}$ and $\PPP: \SSS \to [0;+\infty)$ be a continuous function. Consider the following problem: find a closed Lipschitz curve $\gamma$ on the plane $\mathbb{R}^2$ that encloses a given area $A > 0$ and has the minimum possible integral
\begin{equation}
	\label{problem: main}
	\int\limits_0^{s_1} \PPP(\theta(s))\,ds\to\min,
\end{equation}
where $s$ is the arc length parameter on $\gamma$ (that is, $s_1$ is the length of the curve, and if $\gamma = (x(s), y(s))$, then $x'(s)^2 + y'(s)^2 = 1$ for almost all $s \in [0, s_1]$), and $\theta(s)$ is the angle of the outward normal to $\gamma$ at the point $x(s), y(s)$, i.e., $(x'(s), y'(s)) \perp e_{\theta(s)} = (\cos\theta(s), \sin\theta(s))$. If, for convenience, we assume that the arc length parameter traverses $\gamma$ counterclockwise, then
$$
x'(s) = -\sin \theta(s), \qquad y'(s) = \cos \theta(s).
$$

\begin{remark}
	
Note that we do not impose a convexity requirement on the curve, but merely require that it be rectifiable and closed. We interpret the area constraint via Green's formula: $\frac12\int_0^{s_1}(xy'-x'y)\,ds=A$. Note that we have significantly expanded the class of admissible curves by including in our consideration both nonconvex curves and curves with self-intersections. Thus, Problem \eqref{problem: main} seeks the minimum of the functional in a more general domain as compared with Problem \eqref{F}. Nevertheless, we will prove that the optimal solution to problem~\eqref{problem: main} in this expanded class is ultimately a convex curve without self-intersections; that is, the minima in both problems coincide.
	
\end{remark}

To solve isoperimetric problem~\eqref{problem: main}, we formulate it as an optimal control problem on the curve $(x(t), y(t))$. We will use the function $\operatorname{Arg}(x+iy):\mathbb{C}\setminus 0 \to \SSS$, which takes the form $\operatorname{Arg}(R(\cos\theta+i\sin\theta))=\theta\in \SSS$ for any $R>0$ and $\theta\in \SSS$:
$$
\begin{gathered}
	\int\limits_0^{s_1} \PPP(\operatorname{Arg}(y'-ix'))\,ds \to \min \\
	\frac{1}{2} \int\limits_0^{s_1} (xy' - x'y)\,ds = A \\
	x'^2 + y'^2 = 1 \\
	x(0) = y(0) = x(s_1) = y(s_1) = 0.
\end{gathered}
$$

The duration of motion $s_1$ is equal to the Euclidean length of the curve and is not fixed. Note that in this equation, the isoperimetric condition on the area does not depend on the parametrization of the curve. The minimized functional can also be freed from the requirement of an arc length parameter $\dot{x}^2+\dot{y}^2=1$, which is rather inconvenient. Denote
$$
Q(u,v) \eqdef \mu_\Omega(v,-u) = \sqrt{u^2+v^2}\, \PPP(\operatorname{Arg}(v-iu)).
$$

Then if $(x(t), y(t))$ is some curve for $t \in [0, 1]$, and $s = s(t)$ is its arc length parameter (that is, $s(t) = \int_0^t \sqrt{\dot{x}^2+\dot{y}^2}\,d\tau$), then
$$
\int\limits_0^1 Q(\dot{x}, \dot{y})\,dt =
\int\limits_0^1 \sqrt{\dot{x}^2+\dot{y}^2}\,\PPP(\operatorname{Arg}(\dot{y}-i\dot{x}))\,dt =
\int\limits_0^{s_1} \PPP(\operatorname{Arg}(y'-ix'))\,ds,
$$
since $ds = \sqrt{\dot{x}^2+\dot{y}^2}\,dt$ and $\operatorname{Arg}(\lambda (a+ib)) = \operatorname{Arg}(a+ib)$ for $\lambda > 0$.

Note that the function $Q$ is nonnegative and positively homogeneous. Therefore, it coincides with the Minkowski function of the following set $\Omega_1$, $Q=\mu_{\Omega_1}$:
\begin{equation*}
	%\label{eq: Omega}
\Omega_1 = e^{i\pi/2}\Omega = \Big\{(u, v)\,\Big|\, Q(u,v) \le 1\Big\}.
\end{equation*}
From the definition, it is obvious that the set $\Omega_1$ is a $90^\circ$ rotation of $\Omega$ and hence it is star-shaped, closed, and $0\in\operatorname{int}\Omega_1$.

Thus, problem~\eqref{problem: main} takes an equivalent form:
\begin{equation}
	\label{problem: length minimization}
	\begin{gathered}
		\int\limits_0^1 \mu_{\Omega_1}(u, v) dt \to \min \\
		\dot x=u\qquad \dot y=v\qquad \dot z=\frac12(xv-yu)\\
		x(0) = y(0) = x(1) = y(1) = z(0) = 0\qquad z(1)=A.
	\end{gathered}
\end{equation}
Here $u$ and $v$ are the controls. Note that problem~\eqref{problem: length minimization} is of independent interest for any star-shaped sets $\Omega_1$.

In the particular case when $\Omega_1$ is a compact convex set and $0\in\operatorname{int}\Omega_1$, this problems turns out to be an isoperimetric problem on a Fisler plane. In \cite{Busemann}, the isoperimetric problem on a Finsler (Minkowski) plane was solved for the first time. In \cite{Berestovskii}, it was reformulated as a sub-Finsler problem and solved using Pontryagin's Maximum Principle (PMP). Later, in \cite{Lokut1} and \cite{Lokut2}, this problem, along with a wide class of other sub-Finsler problems, was explicitly solved using convex trigonometry. In this paper, we deal with the situation when $\Omega_1$ is non-convex and non-compact in general.

We will proceed as follows. In this section, we fully analyze the case when the set $\Omega_1$ is convex: we prove the existence of an optimal solution in problem~\eqref{problem: length minimization} and find it using Pontryagin's maximum principle (PMP). Then in the next section we will show that if the set $\Omega_1$ is not convex, then by replacing the set $\Omega_1$ in the problem with $\tilde\Omega_1=\operatorname{conv}\Omega_1$, we obtain
$$
\Omega_1\subset {\tilde\Omega_1}
\,\implies\,
\mu_{\Omega_1} \ge \mu_{\tilde\Omega_1}.
$$
That is, in the new problem, the value of the functional on any curve is no greater than in the original one. An optimal solution exists in the new problem, and we will show that the values of the functionals actually coincide for it, i.e.\ optimal solutions to the problem with non-convex $\Omega_1$ exist and coniside with optimal solutions to the convexified problem.

Let us now write down the condition on $\PPP$ that determines the convexity of $\Omega_1.$

\begin{lemma}
	The set $\Omega_1$ in~\eqref{eq: Omega} is convex and closed if and only if  $\PPP + \PPP'' \ge 0$, where the second derivative is understood in the sense of distributions.
\end{lemma}

This statement is well known; see, for instance, the paper \cite{LP1}, p. 157.

Thus, if the condition of the lemma is met, then the set $\Omega_1$ is convex (and, obviously, closed).
If $\PPP(\theta) \neq 0$ for all $\theta$, then the set of admissible controls $\Omega_1$ is compact.
The case when the set $\Omega_1$ is a convex compact set and $0 \in \operatorname{int} \Omega_1$ has been well studied. In this case, problem \eqref{problem: length minimization} is called an isoperimetric problem on the Finsler plane, see \cite{Lokut1}.

We show that regardless of the compactness of $\Omega_1$, its convexity guarantees the existence of an optimal solution in problem~\eqref{problem: length minimization}.

\begin{theorem}
	\label{thm: existence}
	Assume that the set $\Omega_1$ is convex and $0\in\operatorname{int}\Omega_1$. Then if $\Omega_1$ is not a half-plane or a strip between two parallel lines, then an optimal solution exists in problem~\eqref{problem: length minimization}.
\end{theorem}

\begin{remark}
We will show that if $\Omega_1$ is a half-plane or a strip then the solution does not exist. Further, in this case the infimum of the functional is zero, while otherwise (if $\Omega_1 \ne \mathbb{R}^2$) it is positive (see Section~\ref{sec: non convex})
\end{remark}

\begin{remark}
A more general result, without the assumption of convexity of $\Omega_1$, in a slightly different context, is proved in Theorem \ref{t1} of Section \ref{sec: non convex}.
\end{remark}

\begin{lemma}
	\label{lm: length is bounded}
	Let $\Omega_1\subset\R^2$ be a convex set, $0\in\operatorname{int}\Omega_1$ and $\Omega_1$ contains no straight line. Then for any closed rectifiable curve $(x(t),y(t))$, $t\in[0;1]$, we have
	$$
	\operatorname{length}(x,y)=
	\int_0^1\sqrt{\dot x^2+\dot y^2}\le
	\const\,\int_0^1\mu_{\Omega_1}(\dot x,\dot y)\,dt.
	$$
\end{lemma}

\begin{proof}
	
	First, we show that the polar set $\Omega_1^\circ\subset\R^{2*}$ has a non-empty interior. Indeed, if $\operatorname{int}\Omega_1^\circ=\emptyset$, then the convex set $\Omega^\circ$ is one-dimensional or zero-dimensional, meaning that in any case it is contained in some straight line passing through the origin (since $(0,0)\in\Omega_1^\circ$). This implies that in this case the set\footnote{$\operatorname{cl}$ denotes closure of the set as usual.} $\Omega_1^{\circ\circ}=\operatorname{cl}\operatorname{conv}\Omega_1=\operatorname{cl}\Omega_1$ is the product of some non-empty one-dimensional convex set and a line orthogonal to $\R\Omega_1^\circ$, which contradicts the condition of the lemma. Thus, $\operatorname{int}\Omega_1^\circ\ne\emptyset$.
	
	Further, for any covector $(p,q)\in\Omega_1^\circ$ we have $\mu_{\Omega_1}(\xi,\eta)\ge p\xi+q\eta$ for all $(\xi,\eta)\in\R^2$. Since $(0,0)\in\Omega_1^\circ$, we get
	$$
	(p,q)\in\Omega_1^\circ
	\qquad\Longrightarrow\qquad
	\forall (\xi,\eta)\in\R^2\quad \mu_{\Omega_1}(\xi,\eta)\ge\max\{0,p\xi+q\eta\}.
	$$
	Denote $\dot x(t)=u(t)$ and $\dot y(t)=v(t)$. Then
	$$
	\int_0^1 \max\{0,pu+qv\}\,dt \le \int_0^1\mu_{\Omega_1}(u,v)\,dt \eqdef M.
	$$
	Due to the closedness of the curve $x(0)=x(1)$ and $y(0)=y(1)$, we have $\int_0^1 (pu+qv)\,dt=0$, that is
	$$
	\int_0^1 \max\{0,-pu-qv\}\,dt = -\int_0^1\min\{0,pu+qv\}\,dt =
	\int_0^1\max\{0,pu+qv\} \le M.
	$$
	Thus
	$$
	\int_0^1 |pu+qv|\,dt = \int_0^1 \big(\max\{0,pu+qv\} + \max\{0,-pu-qv\}\big) \,dt \le 2M.
	$$
	
	Next, since the polar set $\Omega_1^\circ$ has a non-empty interior, two non-parallel vectors can be found in it: $(p_i,q_i)\in\Omega_1^\circ$, $i=1,2$, $(p_1,q_1)\nparallel(p_2,q_2)$. For these vectors it holds that
	$$
	\int_0^1 \big(|p_1u+q_1v| + |p_2u+q_2v|\big)\,dt \le 4M.
	$$
	Since the vectors $(p_i,q_i)$ are not parallel to each other, the function $f(\xi,\eta)=|p_1\xi+q_1\eta| + |p_2\xi+q_2\eta|$ is a norm on $\R^2$. All norms on $\R^2$ are equivalent, hence
	$$
	\int_0^1\sqrt{u^2+v^2}\,dt \le \const \int_0^1 \big(|p_1u+q_1v| + |p_2u+q_2v|\big)\,dt \le 4\const M,
	$$
	as required.
\end{proof}

\begin{proof}[Proof of Theorem~\ref{thm: existence}]
	
	If $\Omega_1=\R^2$, then $\mu_{\Omega_1}\equiv0$ and any admissible curve is optimal. Therefore, without loss of generality, we assume that $\Omega_1\ne\R^2$.
	
	Since $0\in\operatorname{int}\Omega_1$, there exists an admissible curve in problem~\eqref{problem: length minimization} on which the functional value is finite. Indeed,
	$$
	0\in\operatorname{int}\Omega_1\quad \Longrightarrow\quad \forall (u,v)\ \mu_{\Omega_1}(u,v)<\infty,
	$$
	that is, on any admissible curve the functional value is finite. The existence of admissible curves is also obvious — for example, a circle on the plane bounding a disk of a given area $A>0$ will do.
	
	Let $u_n,v_n\in L_\infty(0;1)$ be a minimizing sequence of controls. Without loss of generality, $|u_n(t)|^2 + |v_n(t)|^2 \equiv l_n^2$, where $l_n$ is the Euclidean length of the curve $(x_n(t),y_n(t))$ on the plane. The sequence of lengths $l_n$ is bounded due to Lemma~\ref{lm: length is bounded}.

	Due to the boundedness of $l_n$, a weakly$^*$ converging subsequence in $L_\infty(0;1)$ can be selected from the sequence $(u_n,v_n)$. Without loss of generality, we keep the numbering for the subsequence: $(u_n,v_n)\stackrel{\ast}{\rightharpoonup}(u,v)$ in $L_\infty(0;1)$. Thus, for any $t$, the sequences $x_n(t)$ and $y_n(t)$ have limits $x(t)$ and $y(t)$, that is $x_n\to x$ and $y_n\to y$ pointwise. Moreover, $\dot x=u$ and $\dot y=v$. Indeed, for any $t$ due to weak$^*$ convergence $u_n\stackrel{\ast}{\rightharpoonup}u$ we have
	$$
	x(t) = \lim_{n\to\infty} x_n(t) = \lim_{n\to\infty}\int_0^t u_n(\tau)\,d\tau =
	\int_0^t u(\tau)\,d\tau.
	$$
	Therefore $\dot x=u$ and similarly $\dot y=v$. Further, due to the boundedness of $l_n$, for all $n$ the curves $x_n$ and $y_n$ are Lipschitz with the same Lipschitz constant $C=\sup_n l_n$. Consequently, pointwise convergence implies uniform convergence, $x_n\rightrightarrows x$. Whence we immediately obtain that the sequence $z_n$ has a pointwise limit. By similar reasoning $\dot z = \frac12(xv-yu)$.
	
	Thus, the curve $(x(t),y(t),z(t))$ satisfies the boundary conditions of problem~\eqref{problem: length minimization} and the differential equations. It remains to note that the functional in problem~\eqref{problem: length minimization} is lower semi-continuous with respect to weak$^*$ convergence by Tonelli’s Theorem on weak lower semicontinuity, since the function $\mu_{\Omega_1}$ is convex and closed (see \cite{RenardyRogers}).
	
\end{proof}

We will now find the optimal solution to problem \eqref{problem: main} via Pontryagin's maximum principle.

\begin{theorem}
	\label{thm: main}
	Assume that the set $\Omega_1$ is convex, $0\in\operatorname{int}\Omega_1$ and $\Omega_1$ is not a plane, a half-plane, or a strip between two parallel lines. Then the optimal $\gamma$ in problem~\eqref{problem: main} exists by Theorem~\ref{thm: existence} and coincides with the boundary of the polar set $\partial\Omega_1^\circ$ up to a clockwise rotation by $\pi/2$, translation, and dilation by a factor of $\sqrt{A/|\Omega_1^\circ|}$.

    If, additionally, the set $\Omega_1$ is obtained from~\eqref{eq: Omega}, then the optimal curve in problems~\eqref{problem: main} and~\eqref{problem: length minimization} can be parametrized as follows\footnote{Note that by definition $\PPP$ is a lipschitz continuous function and its derivative has countable number of discontinuity points at most. At each point $\theta$ of $\PPP'$ discontinuity, function $\PPP'$ has left and right limits and we understand $\PPP'(\theta)$ as the interval between these limits. So at points $\theta$ of discontinuity of $\PPP'$, the formula below should be viewed as an inclusion.}:
	$$
	\gamma(\theta)=x(\theta) + iy(\theta) = \sqrt{\tfrac{A}{|\Omega_1^\circ|}}\, \big(\PPP(\theta)+i\PPP'(\theta)\big)e^{i\theta} + z_0,
	$$
	where
	$$
	|\Omega_1^\circ| = |\Omega^\circ| = \frac12\int_0^{2\pi} \PPP(\theta)\big(\PPP(\theta)+\PPP''(\theta)\big)\,d\theta,
	$$
	and $z_0=x_0+iy_0\in\mathbb{C}$ is some constant.
\end{theorem}

\begin{proof}
	
	So, given an optimal curve $\gamma$ in problem~\eqref{problem: main}, the corresponding optimal curve in problem~\eqref{problem: length minimization} is constructed by simply reparametrizing time from the interval $s\in[0;s_1]$ to $t\in[0;1]$. So, let the curve $(\hat{x}(t), \hat{y}(t), \hat{z}(t))$ with control $(\hat{u}(t), \hat{v}(t))\in\Omega_1$ be optimal in \eqref{problem: length minimization}. Since reparametrization of the curve changes nothing, we will without loss of generality assume that $\hat u(t)^2+\hat v(t)^2\equiv\const$ for almost all $t$.
	
	According to the PMP, there exist a constant $\lambda_0\in\{0,1\}$ and a Lipschitz curve $(\hat{p}(t), \hat{q}(t), \hat{r}(t)) \neq 0$ (not simultaneously zero) such that for the Pontryagin function
	$$
	H = -\lambda_0\mu_{\Omega_1}(u,v) + pu + qv + \frac{1}{2} r (xv - yu)
	$$
	the Hamiltonian equations hold:
	$$
	\begin{cases} \dot{\hat{p}} = -\hat{H}_x = -\frac{1}{2}\hat{r}\hat{v} \\ \dot{\hat{q}} = -\hat{H}_y = \frac{1}{2}\hat{r}\hat{u} \\ \dot{\hat{r}} = -\hat{H}_z = 0 \end{cases}
	$$
	(that is $\hat{r}(t) = \const$); as well as the maximum principle:
	$$
	\forall (u, v) \in \R^2 \quad
	-\lambda_0\mu_{\Omega_1}(\hat u,\hat v) + \hat{p}\hat{u} + \hat{q}\hat{v} + \frac{1}{2}\hat{r}(\hat{x}\hat{v} - \hat{y}\hat{u}) \ge
	-\lambda_0\mu_{\Omega_1}(u,v)+\hat{p}u + \hat{q}v + \frac{1}{2}\hat{r}(\hat{x}v - \hat{y}u)
	$$
	
	Denote the coefficients for $u$ and $v$ in $H$ by
	$$
	h_1 = p - \frac{1}{2}ry \qquad h_2 = q + \frac{1}{2}rx.
	$$
	Then
	\begin{equation}
		\label{eq: pmp}
		\begin{gathered}
			\begin{cases}
				\dot{h}_1 = \dot{p} - \frac{1}{2}r\dot{y} = -rv \\
				\dot{h}_2 = \dot{q} + \frac{1}{2}r\dot{x} = ru
			\end{cases}\\
			H = - \lambda_0\mu_{\Omega_1}(u,v) + h_1u+h_2v\to\max_{u,v}.
		\end{gathered}
	\end{equation}
	We will act as follows: we will find all solutions to the system~\eqref{eq: pmp} with $u(t)^2+v(t)^2\equiv\const$, and then distinguish among them the one that corresponds to the optimal trajectory.
	
	So, let $(\lambda_0,h_1,h_2,u,v)$ be a solution to the system~\eqref{eq: pmp}. If $\lambda_0=0$, then the maximum of the linear function $h_1u+h_2v$ over $(u,v)\in\R^2$ is achieved only if $h_1\equiv h_2 \equiv 0$. Consequently $\dot h_1\equiv \dot h_2\equiv 0$, meaning $ru\equiv rv\equiv 0$. If $r=0$, then $h_1=p=0$, $h_2=q=0$ and $\lambda_0=0$ by assumption, and we obtain a contradiction with the non-triviality condition of the PMP. Thus, $r\ne 0$. This means $u\equiv v\equiv 0$. Such a solution to the system~\eqref{eq: pmp} cannot correspond to an optimal trajectory, since if $\hat u\equiv\hat v\equiv 0$, then $\hat x\equiv \hat y\equiv \hat z\equiv0$, but $\hat z(1)=A>0$.
	
	So we have $\lambda_0=1$. In this case, the maximum condition in~\eqref{eq: pmp} takes the form $-\mu_{\Omega_1}(u,v)+h_1u+h_2v\to\max_{u,v}$. Consequently, $(h_1,h_2)\in\partial\mu_{\Omega_1}(u,v)$ or equivalently
	\begin{equation}
		\label{eq: u v in subdifferential}
		(u,v)\in\partial \delta_{\Omega_1^\circ}(h_1,h_2),
	\end{equation}
	where the symbol $\partial$ denotes the subdifferential of a convex function, and $\delta_{\Omega_1^\circ}$ is the indicator function of the set $\Omega_1^\circ$, that is $\delta_{\Omega_1^\circ}(h_1,h_2)=0$ if $(h_1,h_2)\in\Omega_1^\circ$ and $\delta_{\Omega_1^\circ}(h_1,h_2)=+\infty$ if $(h_1,h_2)\not\in\Omega_1^\circ$. In particular, if $(h_1,h_2)\not\in\Omega_1^\circ$
    then the maximum in~\eqref{eq: pmp} is not achieved, and if $(h_1,h_2)\in\operatorname{int}\Omega_1^\circ$, then the maximum is achieved at the unique point $(u,v)=(0,0)$. Therefore if $(h_1,h_2)\in\operatorname{int}\Omega_1^\circ$ at some time instance $\tau$, then $(h_1,h_2)\in\operatorname{int}\Omega_1^\circ$ in a neighborhood of $\tau$, and, consequently, $u=v=0$ on an entire time interval. Consequently, $u\equiv v\equiv 0$, which contradicts the positivity of the area $A>0$.
	
	Thus, $(h_1,h_2)\in\partial\Omega_1^\circ$ for all $t$. In this case, the subdifferential $\partial\delta_{\Omega_1^\circ}(h_1,h_2)$ coincides with the normal cone to $\Omega_1^\circ$ at the point $(h_1,h_2)\in\partial\Omega_1^\circ$.
	
	We will now show that $r\ne 0$. Indeed, if $r=0$, then $h_1\equiv\const$ and $h_2\equiv\const$ by virtue of system~\eqref{eq: pmp}. This means, $(u,v)$ does not leave the normal cone to $\Omega^\circ$ at the point $(h_1,h_2)\in\partial\Omega_1^\circ$. Since $\Omega_1$ is not the entire plane, the polar set $\Omega_1^\circ$ is not a single point set, and the normal cone does not coincide with the entire plane at any point. Then by virtue of convexity, the normal cone is contained within some half-space, which contradicts the closedness of the curve $(x(t),y(t))$.
	
	Thus, the point $(h_1(t),h_2(t))$ moves along the boundary of the polar set $\partial\Omega_1^{\circ}$, wherein the case $r>0$ corresponds to counterclockwise motion, and $r<0$ to clockwise motion. Since we assumed without loss of generality that $u^2(t)+v^2(t)\equiv\const$, the point $(h_1(t),h_2(t))$ moves along the boundary of the polar set with constant speed. Note that since $0\in\operatorname{int}\Omega_1$, the polar set $\Omega_1^\circ$ is a compact set, meaning that the polar set $\Omega_1^\circ$ has a finite area $|\Omega_1^\circ|<\infty$, and its boundary $\partial\Omega_1^\circ$ has finite length.
	
	Further, since $r\ne 0$, we have
	$$
	\dot{x} = \dot{h}_2 / r \qquad \text{and} \qquad \dot{y} = -\dot{h}_1 / r,
	$$
	that is, the curves $(x(t), y(t))$ and $(h_2(t), -h_1(t))$ coincide after dilation and translation.
	Since $x(0) = x(T)$ and $y(0) = y(T)$, the curve $(x(t), y(t))$ completes an integer number of revolutions along the boundary of the rotated and dilated polar set $\frac{1}{r} e^{-i\pi/2} \partial \Omega_1^\circ$ (up to translation by a constant vector). However, there cannot be more than 1 revolution, because a Jacobi conjugate point arises after the first revolution (see \cite{Lokut1}). Let us also note that the condition of positive area $A>0$ guarantees that the traversal is completed counterclockwise, that is $r>0$ and, moreover, $|\Omega_1^\circ|=r^2A$. Since the set $\Omega_1$ is not a plane, half-plane, or a strip between parallel lines, the polar set $\Omega_1^\circ$ has a non-empty interior. Indeed, if $|\Omega_1^\circ|=0$, then the polar set $\Omega_1^\circ$ is a one-dimensional convex compact set, i.e., a line segment or a point, and therefore $\Omega_1=\Omega_1^{\circ\circ}$ is a plane, half-plane, or a strip\footnote{By definition, a polar set contains the origin. If $\Omega_1^\circ=\{(0,0)\}$, then $\Omega_1$ is a plane; if $\Omega_1$ is a line segment and the point $(0,0)$ is one of the endpoints of this segment, then $\Omega_1^\circ$ is a half-plane; if $\Omega_1$ is a line segment and the point $(0,0)$ lies interior to this segment, then $\Omega_1$ is a strip between parallel lines.}. Consequently, $|\Omega_1^\circ|\ne 0$ and $r^2=|\Omega_1^\circ|/A$.
	
	Let us compute $\partial \Omega_1^\circ$ in terms of the original function $\PPP(\theta)$. Every point $(\xi, \eta) \in \partial \Omega_1$ corresponds to a set of supporting covectors, that is, points $(p, q) \in \partial \Omega_1^\circ$, such that $p\xi + q\eta = 1$ and $\forall (u, v) \in \partial \Omega_1$ it holds that $pu + qv \le 1$. Since
	$$ \partial \Omega_1 = \left\{ \left( \frac{\cos\varphi}{\PPP(\varphi-\pi/2)} ; \frac{\sin\varphi}{\PPP(\varphi-\pi/2)} \right) \right\}, $$
	the covector $(p, q)$ lies on the boundary $\partial \Omega_1^\circ$ if
	$$ \max_{\varphi : \PPP(\varphi-\pi/2) > 0} \frac{p \cos\varphi + q \sin\varphi}{\PPP(\varphi-\pi/2)} = 1. $$
	
	At the point of the maximum, obviously, $\PPP(\varphi-\pi/2) \neq 0$, therefore
	$$ \begin{cases} p \cos\varphi + q \sin\varphi = \PPP(\varphi-\pi/2) \\ (-p \sin\varphi + q \cos\varphi)\PPP(\varphi-\pi/2) - (p \cos\varphi + q \sin\varphi)\PPP'(\varphi-\pi/2) = 0 \end{cases} $$
	That is $-p \sin\varphi + q \cos\varphi = \PPP'(\varphi-\pi/2)$.
	Meaning, the formulas
	\begin{align*}
		p &= \PPP(\varphi-\pi/2)\cos\varphi - \PPP'(\varphi-\pi/2)\sin\varphi \\
		q &= \PPP(\varphi-\pi/2)\sin\varphi + \PPP'(\varphi-\pi/2)\cos\varphi
	\end{align*}
	define $\partial \Omega_1^\circ$. Let us compute the area of the polar set:
	$$
	|\Omega_1^\circ| = \frac12\int_0^{2\pi} (pq'-p'q)\,d\varphi.
	$$
	Denote $z=p+iq$. Then $z=e^{i\varphi}(\PPP+i\PPP')$ and $p'+iq'=z'=ie^{i\varphi}(\PPP+\PPP'')$. Hence $pq'-p'q=\operatorname{Im}\bar zz' = \PPP(\PPP+\PPP'')$ and
	$$
	|\Omega_1^\circ| = \frac12\int_0^{2\pi} \PPP(\varphi-\tfrac\pi2)\big(\PPP(\varphi-\tfrac\pi2)+\PPP''(\varphi-\tfrac\pi2)\big)\,d\varphi =
	\frac12\int_0^{2\pi} \PPP(\theta)(\PPP(\theta)+\PPP''(\theta))\,d\theta
	$$
	
	The optimal curve can be parameterized via $\varphi$, since it coincides with $\frac{1}{r}e^{-i\pi/2}\partial\Omega_1^\circ$:
	\begin{align*}
		x(\varphi) &= \frac{1}{r} \left( \PPP(\varphi-\pi/2)\sin\varphi + \PPP'(\varphi-\pi/2)\cos\varphi \right) + x_0 \\
		y(\varphi) &= \frac{1}{r} \left( -\PPP(\varphi-\pi/2)\cos\varphi + \PPP'(\varphi-\pi/2)\sin\varphi \right) + y_0
	\end{align*}
	
	If we denote $\theta = \varphi - \pi/2$, we obtain
	\begin{align*}
		x(\theta) &= \frac{1}{r} \left( \PPP(\theta)\cos\theta - \PPP'(\theta)\sin\theta \right) + x_0\\
		y(\theta) &= \frac{1}{r} \left( \PPP(\theta)\sin\theta + \PPP'(\theta)\cos\theta \right) + y_0,
	\end{align*}
	which coincides with the formula stated in the theorem formulation.
	
\end{proof}

\begin{remark}
	Note that in problem~\eqref{problem: main} we sought a minimum among all closed curves, but the optimal curve turned out to be convex.
\end{remark}

\section{The case of (generally) nonconvex \texorpdfstring{$\Omega$}{Ω}}
\label{sec: non convex}

%{\red From now on, we slightly change the notation for $S^1$: now it is the unit circumference on the plane.}

In the previous section we studied the particular case of convex $\Omega$. In this section we consider the general case. We prove the theorem of existence for (generally) nonconvex $\Omega$ (Theorem \ref{t1}) and show that the solution in the general case can be obtained through convexification of $\Omega$ (Theorem \ref{t2}).

These theorems  will be proved using the representation of the resistance functional in terms of the surface area measure.

\begin{theorem}\label{t1}
(a) Let $\PPP > 0$ on an interval $(\theta_1,\theta_2)$ greater than $\pi$, $\theta_2-\theta_1>\pi$. (It may happen, in particular, that $\PPP$ is positive everywhere.) Then the minimum of problem~\eqref{F} exists, and its value is positive.

(b) Let $\PPP > 0$ on an interval $(\theta_0,\theta_0+\pi)$, and $\PPP(\theta_0) = 0 = \PPP(\theta_0+\pi)$. Then the infimum is zero, and it is not attained.

(c) Suppose that any interval $(\theta,\theta+\pi)$ contains a zero of $\PPP$ (and therefore $\PPP$ has at least 3 zeros).
%, and the convex hall of the zero set contains the origin in its interior).
Then the minimum exists and is equal to zero. A minimizer is unique if and only if $\PPP$ contains exactly 3 zeros.
\end{theorem}

\begin{remark}
It is instructive to compare this theorem with Theorem \ref{thm: existence} from the previous section. Theorem \ref{t1} is valid for all $\PPP$, while Theorem \ref{thm: existence} holds when the corresponding set $\Omega=\{(u,v)\,\Big|\, \sqrt{u^2+v^2}\,\PPP(\operatorname{Arg}(v-iu)) \le 1 \}$ is convex. On the other hand, Theorem \ref{thm: existence} seeks for a solution in a class of closed curves and asserts that the optimal curve boulds a convex set, while Theorem \ref{t1} deals with convex sets $B$ and their boundaries $\partial B.$

Note also that the conditions of non-existence of solution are in agreement in Theorems \ref{thm: existence} and \ref{t1}. Indeed, a convex set $\Omega$ containing 0 in its interior is a half-plane or a strip if and only if condition (b) is satisfied.
\end{remark}

\begin{proof}
It is convenient to work with the function $\PP : S^1 \to \mathbb{R}$ defined by $\PP(\cos\theta,\sin\theta) = \PPP(\theta)$ and with the functional $\FF$ in the form \eqref{Fmeas}. In (a), $\PP > 0$ on an arc of the circumference $S^1$ greater than $\pi$. In (b), $\PP > 0$ on an open semicircumference with the endpoints $e= (\cos\theta_0, \sin\theta_0)$ and $-e$. In (c), any open semicircumference contains a zero of $\PP$ (and therefore $\PP$ has at least 3 zeros, and the convex hull of the zero set contains the origin in its interior).

(a) Let us first show that a sequence of bodies $B_m$ minimizing $\FFF$ has uniformly bounded perimeters.% or equivalently, the sequence $\{ \nu_{B_m}(S^1) \}$ is bounded.

First we consider the particular case when $\PP$ is always positive. Then there exists a positive constant $c$ such that $\PP \ge c$. It follows that
$$
c|\partial B_m|\, =\, \int_{S^1} c\, d\nu_{B_m}(n)\, \le\, \FF(\nu_{B_m})\, =\, \FFF(B_m).
$$
Since the sequence $\FFF(B_m)$ converges, and therefore, is bounded above, it follows that $|\partial B_m|$ is also bounded above.

In the general case take a closed arc $\gamma \subset S^1$ greater than $\pi$ on which $\PP$ is positive. There exists a positive constant $c$ such that $\PP>c$ on $\gamma$.
Denote $\gamma' = S^1 \setminus \gamma$ and take a vector $e \in S^1$ and a positive constant $k$ such that $\langle e,\, n \rangle \ge k$ for all $n \in \gamma'$. It follows that for $\nu = \nu_{B_m}$
$$
\Big\langle e,\ \int_{\gamma'} n\, d\nu(n) \Big\rangle\ =\
\int_{\gamma'} \langle e, n \rangle\, d\nu(n)\ \ge\ k\, \nu(\gamma'),
$$
and since $\big\langle -e,\ \int_{\gamma} n\, d\nu(n)\big\rangle = \big\langle e,\ \int_{\gamma'} n\, d\nu(n)\big\rangle$, we have
$$
\nu(\gamma)\ \ge\ \int_{\gamma} \langle -e, n \rangle\, d\nu(n)\ =\
\Big\langle -e,\ \int_{\gamma} n\, d\nu(n) \Big\rangle\ \ge\ k\, \nu(\gamma').
$$
It follows that the perimeter of the $m$-th body, $|\partial B_m|$, satisfies
$$
|\partial B_m|\, =\, \nu(S^1)\, =\, \nu(\gamma) +\nu(\gamma')\, \le \, \Big( 1 + \dfrac{1}{k}\Big)\, \nu(\gamma)\, \le\, \Big( 1 + \dfrac{1}{k}\Big)\,\frac{1}{c}\, \FFF(B_m),
$$
and therefore, is bounded above.

Hence the diameters of $B_m$ are uniformly bounded. Making translations if necessary, one can assume that all bodies are contained in a bounded set. By the Blaschke selection theorem, there exists a subsequence, which for convenience will be denoted in the same way, $B_{m}$, converging in the Hausdorff sense to a convex body $B_*$. Then the areas of $B_{m}$ converge to the area of $B_{*}$ (since both $B_m$ and $B_*$`are convex), hence $|B_{*}|=A.$ It is known that Hausdorff convergence of a sequence of convex bodies implies weak convergence of their surface area measures (see, e.g., Theorem 4.2.1 in the book \cite{S}), that is, $\nu_{B_{m}} \to \nu_{B_*}$ as $k\to\infty$. Hence $\FF(\nu_{B_*}) = \lim_{m\to\infty} \FF(\nu_{B_{m}})$ is the minimum of the functional, and so, $\nu_{B_{*}}$ is a minimizer to problem \eqref{F} and $B_*$ is a minimizer to problem \eqref{nu}.
	\vspace{2mm}
	
(b) The zero infimum is attained on a sequence of rectangles whose side lengths go to zero and to infinity, with their product equal to $A$. The outward normals to the longer sides are equal to $e=(\cos\theta_0,\sin\theta_0)$ and $-e$.

The surface area measure of any admissible body $B$ cannot be supported in $\{ e, -e \}$, since otherwise $B$ is a segment and hence has zero area. It follows that both open semicircumferences bounded by $e$ and $-e$, say $\gamma$ and $\gamma'$, have positive $\nu_B$-measure. Suppose that $\PP >0$ on $\gamma$; then
$$
\FFF(B)\, = \, \FF(\nu_B) \, \ge\, \int_{\gamma} \PP (n)\, d\nu_B(n)\, >\, 0.
$$
Therefore, the zero infimum is not attained on a body.
	\vspace{2mm}

(c) We have $\FF(\nu_{B}) = 0$ if and only if the support of $\nu_B$ is contained in the zero set of $\PP$. Take 3 points, say $e_1,\, e_2$ and $e_3$ from this set so that the triangle with vertices at these points contains the origin in its interior. Then the triangle $B$ with area $A$ and the outward normals to the sides equal to $e_1,\, e_2$, and $e_3$ satisfies $\FF(\nu_B) = 0$. If the zero set of $\PP$ coincides with this triple of vectors, the minimizer is uniquely determined; otherwise a set of vectors from the zero set of $\PP$ containing the origin in its interior can be chosen in different ways, for example a triple and a quadruple of vectors, which generate a triangle and a quadrangle with zero value of the functional.
\end{proof}

From now till the end of this section we assume that $\PPP > 0$ on an interval greater than $\pi$. According to item (a) of Theorem \ref{t1}, the minimum of $\FFF$ exists and is positive. If $\PPP$ takes zero at some points, denote by $\gamma$ the greatest interval where $\PPP$ is positive, and denote by $\theta_1$ and $\theta_2$ its endpoints. We have $\PPP(\theta_1) = 0 = \PPP(\theta_2)$; it may happen that $\theta_1 = \theta_2$. If $\PPP>0$ everywhere then $\theta_1$ and $\theta_2$ are not defined.

Let us define the set $\Omega \subset \mathbb{R}^2$ by the inequality $r \PPP(\theta) \le 1$ in polar coordinates $r,\, \theta$, where $x=r\cos\theta,\, y=r\sin\theta$. Equivalently, $\Omega$ is the star-shaped set on the plane such that $\PPP(\theta) = \inf \Big\{ \dfrac{1}{r} : (r,\theta) \in \Omega \Big\}$ for all values of $\theta$. The set $\Omega$ coincides with the one defined in Section \ref{sec: convex}.
%%%%%up to the rotation by $\pi/2$.

Let $\tilde\Omega$ be the convex hull of $\Omega$, $\tilde\Omega = \conv \Omega$. Denote the function $\tilde{\PPP}$ on $\SSS$ by
$$
\tilde \PPP(\theta) = \inf \Big\{ \dfrac{1}{r} : (r,\theta) \in \tilde\Omega \Big\},
$$
and define the functional $\tilde{\FFF}$ by
\begin{equation}\label{tildeF}
\tilde \FFF(B) \ =\ \int_{\partial B} \tilde \PPP(\theta_\xi)\, d\xi.
\end{equation}
Note that
%$\tilde {\PPP}\le \PPP$, and therefore, $\tilde \FFF(B) \le \FFF(B)$ for all $B$:
$$
    \Omega\subseteq\tilde\Omega
    \quad\implies\quad
    \forall\theta\ \ \PPP(\theta)\ge\tilde \PPP(\theta)\ge 0
    \quad\implies\quad
    \forall B\ \ \FFF(B)\ge\tilde \FFF(B)\ge 0.
$$

Recall that a point $\xi \in \tilde\Omega$ is called an {\it extreme} point of the convex set $\tilde\Omega$, if it is not an interior point of a line segment contained in $\tilde\Omega$. The set of all extreme points of $\tilde\Omega$ is denoted as $\extr(\tilde\Omega)$. One easily sees that $\extr(\tilde\Omega) \subset \partial\Omega$ and $\PP(v/|v|) = \tilde\PP(v/|v|)$ for all $v \in \extr(\tilde\Omega)$.

Let us define the set $\KKK = \KKK_{\Omega} \subset S^1$ as follows:
$$
\KKK \eqdef \Big\{ \frac{v}{|v|}\,:\ v \in \extr(\tilde\Omega) \Big\}\subset S^1, \quad \text{if} \ \, \tilde\Omega \, \ \text{is bounded},
$$
and
$$
\KKK \eqdef \Big\{ \frac{v}{|v|}\,:\ v \in \extr(\tilde\Omega) \Big\} \cup \{ e_1, e_2\}, \quad  \text{if} \, \ \tilde\Omega \ \ \text{is unbounded},
$$
where $e_1 \eqdef e_{\theta_1},\, e_2 \eqdef e_{\theta_2}.$
In other words, $\KKK$ is the central projection of the set of extreme points of $\tilde\Omega$ on the circle $S^1$, if $\tilde\Omega$ is bounded, and the projection plus the vectors $e_1$ and $e_2$, otherwise.

Since $\PP(e_i) = \tilde\PP(e_i) = 0,\, i=1,\,2,$ one concludes that $\PP=\tilde \PP$ on $\KKK.$ Further, using that the convex hull of the set $\tilde\Omega$ coincides with itself, it is easy to check that $\KKK = \KKK_{\Omega}$ coincides with~$\KKK_{\tilde\Omega}$.

\begin{remark}
The set $\KKK \subset S^1$ is compact. Indeed, this follows easily from the fact that non-extreme points form an open subset of the boundary $\partial\Omega$. If $\Omega$ is strictly convex (hence its boundary does not contain line segments) and is either bounded or contains only one ray with the vertex at the origin (correspondingly, $\PP$ is either positive or has only one zero),  then $\KKK = S^1$. Otherwise, the complement $S^1 \setminus \KKK$ is the disjoint union of finitely or countably many open circular arcs. Each arc corresponds to a dimple or a line segment on the boundary of $\Omega$. Additionally, if $\Omega$ contains more than one ray with the vertex at the origin (and therefore, $e_1 \ne e_2$) then one of the arcs is $(e_1,\, e_2)$.
\end{remark}

\begin{theorem}\label{t2}
Let $B$ be a solution to problem \eqref{F}. Then

(a)	the surface area measure of $B$ is supported in $\KKK$;

(b) the minimum of $\FFF$ coincides with the minimum of $\tilde \FFF$,
$$\min_{|B|=A} \FFF(B) = \min_{|B|=A} \tilde \FFF(B),$$
and the sets of solutions for $\FFF$ and for $\tilde \FFF$ coincide.
\end{theorem}

\begin{proof}
(a) Let $B$ be an optimal body. We need to show that

(i) If $\partial\tilde\Omega$ contains a segment $[v_1,\, v_2]$ then the $\nu_B$-measure of the arc $(v_1/|v_1|,\, v_2/|v_2|)$ is zero.

(ii) If $\tilde\Omega$ is unbounded and $e_1 \ne e_2$ then the $\nu_B$-measure of the arc $(e_1, e_2)$ is zero.

  %(iii) Suppose that $\partial\tilde\Omega$ contains a ray (and therefore $\tilde\Omega$ is unbounded and the director vector of the ray is $e_i$, $i=1$ or 2), and let $v_0$ be the vertex of the ray. Then the $\nu_B$-measure of the arc $(v_0/|v_0|,\, e_i)$ is zero.
\vspace{2mm}

(i) Take a maximal segment contained in $\partial\tilde\Omega$. We will consider separately the cases (i$_1$) when it is bounded and (i$_2$) when it is unbounded, and hence, semi-infinite.

(i$_1$) Let $[v_1,\, v_2]$ be a maximal segment contained in $\partial\tilde\Omega$, and let the line of support to $\tilde\Omega$ containing this segment be given by $\langle c, v \rangle = 1$, where $c$ is a nonzero vector. Denote $r_i \eqdef |v_i|$ and $\epsilon_i \eqdef v_i/|v_i|,\, i=1,\,2$. Since $v_i \in \partial\Omega$, we have $r_i \PP(\epsilon_i) =1$ and $r_i \langle c, \epsilon_i \rangle = 1$, hence
$$
\PP(\epsilon_1) = \langle c, \epsilon_1 \rangle \quad \text{and} \quad \PP(\epsilon_2) = \langle c, \epsilon_2 \rangle.
$$
\begin{figure}[ht]
	\begin{center}
		\includegraphics[width=0.5\textwidth]{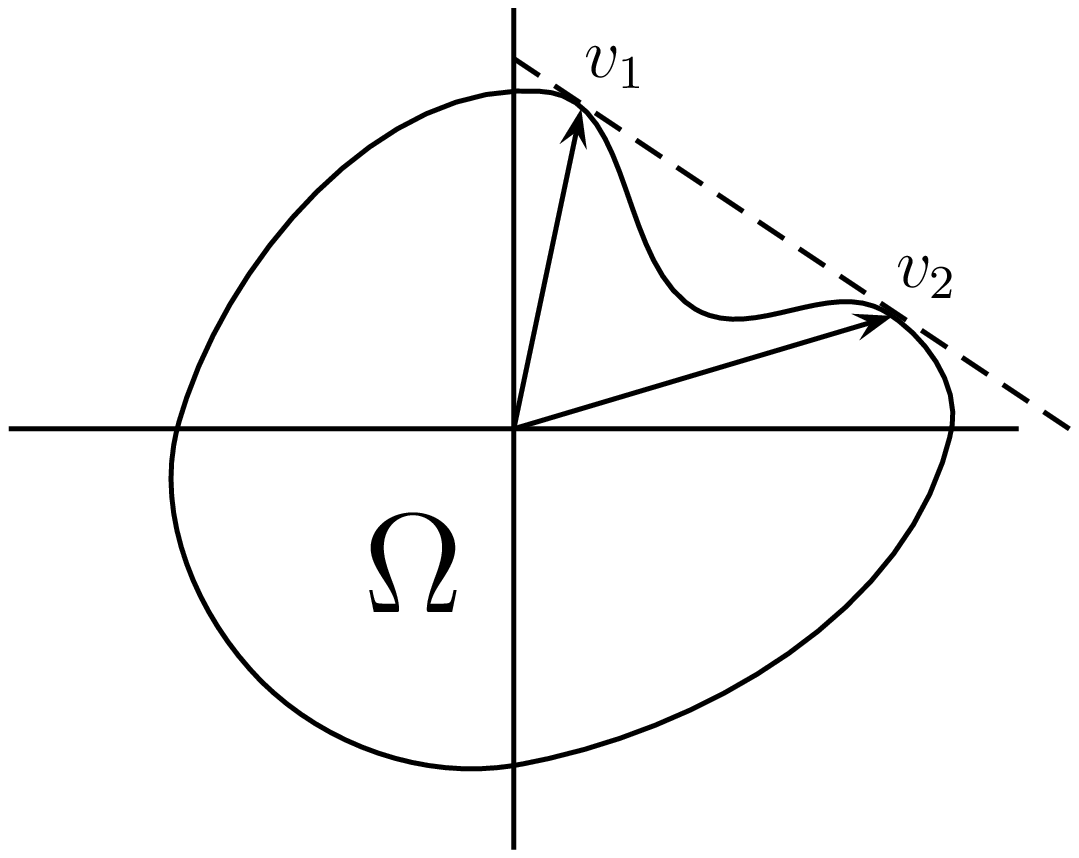}
		\caption{Bounded nonconvex $\Omega$: the case of one dimple.}
		\label{fig:Omega}
	\end{center}
\end{figure}
On the other hand, for each intermediate point $v=re,\, r \ge 0,\, |e|=1$ on the interval $(v_1,\, v_2)$ we have $r \PP(e) \ge 1$ and $r \langle c, e \rangle = 1$,
hence
$$
\PP(e) \ge \langle c, e \rangle \quad \text{for all} \, \ e \in S^1 \, \ \text{between} \, \ \epsilon_1 \, \ \text{and} \, \ \epsilon_2.
$$

Suppose that $\nu_B((\epsilon_1, \epsilon_2)) > 0$. Draw two support lines to $B$ with the outward normals $\epsilon_1$ and $\epsilon_2$, and let $\hat B$ be the new body bounded by these lines and the remaining part of the boundary of $B$; see Fig.~\ref{fig-Body}. In other words, the part of $\partial\Omega$ with the outward normals in $(\epsilon_1, \epsilon_2)$ is substituted by two segments with the outward normals $\epsilon_1$ and $\epsilon_2$. Let $\nu_{\epsilon_1 \epsilon_2}$ be the surface area measure of that part of $\partial\Omega$, and $\lambda_1>0$ and $\lambda_2>0$ be the lengths of the substituting segments.

\begin{figure}[ht]
	\begin{center}
		\includegraphics[width=0.5\textwidth]{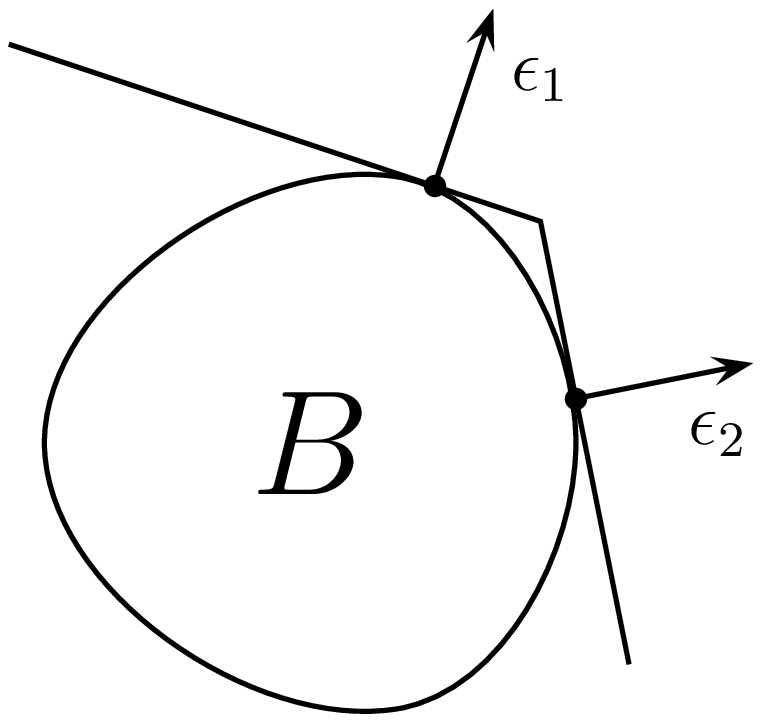}
		\caption{The bodies $B$ and $\tilde B$.}
		\label{fig-Body}
	\end{center}
\end{figure}

The area of $\hat B$ is greater than the area of $B$,\, $|\hat B| > |B|$. Let us show that $\FFF(\hat B) \le \FFF(B)$, and therefore, $\dfrac{\FFF(\hat B)}{\sqrt{|\hat B|}} \le \dfrac{\FFF(B)}{\sqrt{|B|}}$,
in a contradiction with optimality of $B$.

We have $\nu_{\hat B} - \nu_B = \lambda_1 \delta_{\epsilon_1} + \lambda_2 \delta_{\epsilon_2} - \nu_{\epsilon_1 \epsilon_2}$, and
$$
\int_{S^1} n\, d\nu_{B}(n)\ =\ 0 \ =\ \int_{S^1} n\, d\nu_{\hat B}(n) \ \implies \
$$ $$
\ \implies \ \int_{S^1} \langle c, n \rangle \, d(\nu_{\hat B} - \nu_B)(n)\, =\, \lambda_1 \langle c, \epsilon_1 \rangle + \lambda_2 \langle c, \epsilon_2 \rangle - \int_{S^1} \langle c, n \rangle \, d\nu_{\epsilon_1\epsilon_2}(n) = 0,
$$
hence
$$
\FFF(\hat B) - \FFF(B)\, =\, \FF(\nu_{\hat B}) - \FF(\nu_B)\, =\,
$$ $$
\, =\, \lambda_1 \PP(\epsilon_1) + \lambda_2 \PP(\epsilon_2) - \int_{S^1} \PP(n)\, d\nu_{\epsilon_1 \epsilon_2}(n)
$$ $$
\, \le \, \lambda_1 \langle c, \epsilon_1 \rangle + \lambda_2 \langle c, \epsilon_2 \rangle - \int_{S^1} \langle c, n \rangle\, d\nu_{\epsilon_1 \epsilon_2}(n)\, =\, 0.
$$
The contradiction implies that the support of $\nu_B$ does not contains points of $(\epsilon_1,\, \epsilon_2)$.
\vspace{2mm}

(i$_2$) Consider a maximal semi-infinite segment contained in $\partial\tilde\Omega$. Let $v_0 = r_0 \epsilon_0\, (r_0>0,\, |\epsilon_0| = 1)$ be its vertex, and assume that its director vector is $e_1$ (the case of $e_2$ is completely similar). The straight line containing this half-line is given by $\langle e_1^\perp, v \rangle = \langle e_1^\perp, v_0 \rangle$. Denoting $v=rn,\, r >0,\, |n|=1$, one can write down the equation of the straight line as follows
\begin{equation}\label{eq1}
r \langle e_1^\perp, n \rangle \ = \ \langle e_1^\perp, v_0 \rangle.
\end{equation}
Since $v_0 \in \partial\Omega$, we have $r_0 \PP(\epsilon_0) = 1$, hence
\begin{equation}\label{eq2}
\PP(\epsilon_0) = \dfrac{\langle e_1^\perp, \epsilon_0 \rangle}{\langle e_1^\perp, v_0 \rangle}
\end{equation}
 and $r \PP(n) \ge 1$; hence by \eqref{eq1} one obtains
\begin{equation}\label{eq3}
\PP(n) \ge \dfrac{\langle e_1^\perp, n \rangle}{\langle e_1^\perp, v_0 \rangle}
 \quad \text{for all} \, \ n \in S^1 \, \ \text{between} \, \ \epsilon_0 \, \ \text{and} \, \ e_1.
\end{equation}
Again, suppose that $\nu_B((\epsilon_0, e_1)) > 0$ and draw two support lines to $B$ with the outward normals $\epsilon_0$ and $e_1$. Let $\hat B$ be the new body bounded by these lines and the remaining part of the boundary of $B$. That is, the part of $\partial\Omega$ corresponding to the outward normals between $\epsilon_0$ and $e_1$ is substituted by two segments with the outward normals $\epsilon_0$ and $e_1$. Let $\nu_{\epsilon_0 e_1}$ be the surface area measure of this part of $\partial\Omega$, and $\lambda_0 > 0$ and $\lambda_1 > 0$ be the lengths of these segments.

We have $\nu_{\hat B} - \nu_B = \lambda_0 \delta_{\epsilon_0} + \lambda_1 \delta_{e_1} - \nu_{\epsilon_0 e_1}$, and
$$
\int_{S^1} n\, d\nu_{B}(n)\ =\ 0 \ =\ \int_{S^1} n\, d\nu_{\hat B}(n) \ \implies \
$$ $$
\ \implies \ \int_{S^1} \langle e_1^\perp, n \rangle \, d(\nu_{\hat B} - \nu_B)(n)\, =\, \lambda_0 \langle e_1^\perp, \epsilon_0 \rangle + \lambda_1 \langle e_1^\perp, e_1 \rangle - \int_{S^1} \langle e_1^\perp, n \rangle \, d\nu_{\epsilon_0 e_1}(n)
$$ $$
\, =\, \lambda_0 \langle e_1^\perp, \epsilon_0 \rangle\, -\, \int_{S^1} \langle e_1^\perp, n \rangle \, d\nu_{\epsilon_0 e_1}(n)
= 0,
$$
hence using \eqref{eq2} and \eqref{eq3},
$$
\FFF(\hat B) - \FFF(B)\, =\, \FF(\nu_{\hat B}) - \FF(\nu_B)\, =\,
$$ $$
\, =\, \lambda_0 \PP(\epsilon_0)\, +\, \lambda_1 \PP(e_1) - \int_{S^1} \PP(n)\, d\nu_{\epsilon_0 e_1}(n)
$$ $$
\, =\, \lambda_0 \PP(\epsilon_0)\, -\, \int_{S^1} \PP(n)\, d\nu_{\epsilon_0 e_1}(n)
$$ $$
\le \ \dfrac{1}{\langle e_1^\perp, v_0 \rangle}\, \left(
\lambda_0 \langle e_1^\perp, \epsilon_0 \rangle\, -\, \int_{S^1} \langle e_1^\perp, n \rangle\, d\nu_{\epsilon_0 \epsilon_2}(n) \right)\, =\, 0.
$$
On the other hand, the area of $\hat B$ is greater than the area of $B$. Thus, $B$ is not optimal.
\vspace{2mm}

(ii) Suppose that $\nu_B((e_1, e_2)) > 0$ and draw two support lines to $B$ with the outward normals $e_1$ and $e_2$. Let $\hat B$ be the new body bounded by these lines and the remaining part of the boundary of $B$. That is, the part of $\partial\Omega$ corresponding to the outward normals in $(e_1, e_2)$ is substituted by two segments with the outward normals $e_1$ and $e_2$. Let $\nu_{e_1 e_2}$ be the surface area measure of that part of $\partial\Omega$, and $\lambda_1>0$ and $\lambda_2>0$ be the lengths of the substituting segments.

The area of $\hat B$ is greater than the area of $B$,\, $|\hat B| > |B|$. On the other hand, $\FFF(\hat B) \le \FFF(B)$, in a contradiction with optimality of $B$. Indeed, since $\PP(e_1) = 0 = \PP(e_2)$, we have
$$
\FFF(\hat B) - \FFF(B)\, =\,
\FF(\nu_{\hat B}) - \FF(\nu_B)\,=\,
\lambda_1 \PP(e_1) + \lambda_2 \PP(e_2) - \int_{S^1} \PP(n)\, d\nu_{e_1 e_2}(n)\, \le \, 0.
$$
%\vspace{2mm}

(b) According to claim (a), problem \eqref{F} can be stated, in terms of measures, as follows: find a measure supported in $\KKK$ inducing a body of area $A$ that  minimizes the functional
$$ \int_{S^1} \PP(n)\, d\nu(n). $$
Since $\KKK = \KKK_{\Omega}$ coincides with $\KKK_{\tilde\Omega}$ and $\PP=\tilde\PP$ on $\KKK$, the problems for the functions $\PP$ and $\tilde\PP$ (equivalently, for the functionals $\FFF$ and $\tilde \FFF$) are identical, and therefore their sets of solutions coincide.
\end{proof}

\section{Application to problems of least resistance}\label{sec application}

{\bf 5.1.} \
Let $\PPP$ be determined by (see formula \eqref{P} in Section \ref{sec1})
\begin{equation}\label{S5P}
\PPP(\theta)\, = \, \dfrac{1}{2}\int_{-\pi/2}^{\pi/2} \cos^3 \varphi\, f(\theta-\varphi)\, d\varphi,
\end{equation}
where the factor $1/(2a)$ is substituted with $1/2$ for the further convenience. It is easy to check that $\PPP$ is a $C^3$ function, and
\begin{equation}\label{P+P''}
\PPP(\theta) + \PPP''(\theta)\, =\, - \int_{-\pi/2}^{\pi/2} \cos 3\varphi\, f(\theta-\varphi)\, d\varphi
\end{equation}
and
\begin{equation}\label{PP''}
\big(\PPP(\theta) + \PPP''(\theta) \big)'\, =\, 3\int_{-\pi/2}^{\pi/2} \sin 3\varphi\, f(\theta-\varphi)\, d\varphi.
\end{equation}

Assume that $f$ is an even monotone non-increasing function, which is positive in the interior of a non-degenerate interval $[-k,\, k]$ $(0 < k \le \pi)$ and zero outside this interval. (It may happen that $f>0$ everywhere, and so, $k=\pi$.)

Denote
$$
K \eqdef \int_{0}^{\pi/2} \cos 3\varphi\, f(\varphi)\, d\varphi.
$$

Determine the following conditions which mean, in a sense, smallness of oscillations.
\vspace{2mm}

{\bf C1.}\ Either $K > 0$, or $K=0$ and $k < \pi/2$.

{\bf C2.}\ $k < \pi/2$.

\begin{remark}
The values $K$ and $k$ are related to intensity of the body's oscillations, but in different aspects. The inequality $K > 0$ means that rotations by angles between 0 and $\pi/6$ (where $\cos 3\varphi > 0$), in a sense, dominate as compared with rotations by angles between $\pi/6$ and $\pi/2$ (where $\cos 3\varphi < 0$). The inequality $k < \pi/2$ means that the body cannot rotate by an angle greater than $\pi/2$. We will show (see Theorem \ref{t_cond} and Corollary \ref{coro}) that conditions C1 and C2 are equivalent to existence of singularities, respectively, at the front and back points of the optimal body' boundary).
\end{remark}

Below we will prove (see Theorem~\ref{t_cond} and Corollary~\ref{coro}) that the set $\Omega$ may have 4 possible shapes (see Fig.~\ref{fig-cor}). Recall that the polar set $\Omega^\circ$ is the optimal body (up to scaling). In the case (i), $\partial\Omega^\circ$ is $C^1$; in the case (ii), $\partial\Omega^\circ$ has one singular point at the top; in the case (iii), $\partial\Omega^\circ$ has one singular point at the bottom; in the case (iv), $\partial\Omega^\circ$ has two singular points --- one at the top and one at the bottom.

\begin{figure}[ht]
	\begin{center}
		\includegraphics[width=0.5\textwidth]{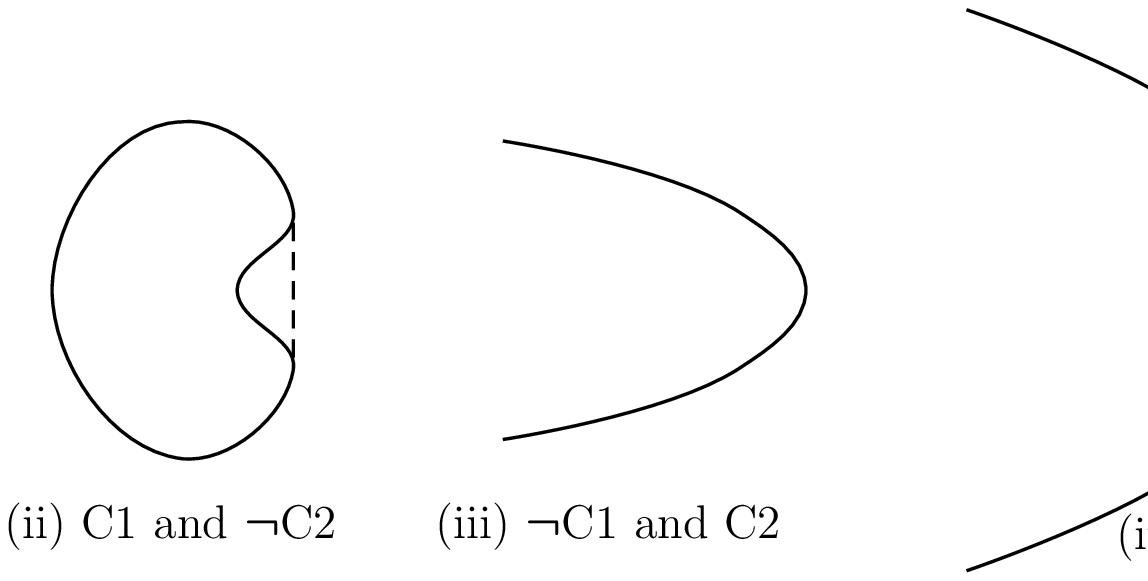}
	\caption{The set $\Omega$ in the cases (i), (ii), (iii), and (iv).}
		\label{fig-cor}
	\end{center}
\end{figure}

Recall that $\SSS = \mathbb{R}/2\pi\mathbb{Z}$ and denote $\hat\SSS \eqdef \SSS \setminus \{0,\, \pi,\, \pm\pi/6,\, \pm(k - \pi/6) \}$. That is, $\hat\SSS$ is $\SSS$ minus the set of (at most) 6 points.

\begin{theorem}\label{t_cond}
(i)\ If $K \le 0$ and $k \ge \pi/2$ (i.e., $\nneg$C1 and $\nneg$C2), then $\PPP>0$ on $\SSS \setminus \{\pi\}$ and $\PPP+\PPP''>0$ on $\hat\SSS$.

(ii)\ If $K > 0$ and $k \ge \pi/2$ (i.e., C1 and $\nneg$C2), then $\PPP>0$ on $\SSS \setminus \{\pi\}$, and $\PPP+\PPP'' \le 0$ in $[-\theta_*,\, \theta_*]$ and $\PPP+\PPP''> 0$ on $\hat\SSS \setminus [-\theta_*,\, \theta_*]$, for some $0 < \theta_* \le \pi/6$.

(iii)\ If $K < 0$ and $k < \pi/2$ (i.e., $\nneg$C1 and C2), then $\PPP>0$ on $(-\pi/2-k,\, \pi/2+k)$ and $\PPP+\PPP'' >0$ on $\hat\SSS \cap (-\pi/2-k,\, \pi/2+k)$, and $\PPP=0$ on $\SSS \setminus (-\pi/2-k,\, \pi/2+k)$.

(iv)\ If $K \ge 0$ and $k < \pi/2$ (i.e., C1 and C2), then $\PPP>0$ on $(-\pi/2-k,\, \pi/2+k)$,\, $\PPP=0$ on $\SSS \setminus (-\pi/2-k,\, \pi/2+k)$,\, $\PPP+\PPP''\le 0$ on $[-\theta_*,\, \theta_*]$ (for some $0 < \theta_* \le \pi/6$), and $\PPP+\PPP''> 0$ on $\hat\SSS \cap (-\pi/2-k,\, \pi/2+k) \setminus [-\theta_*,\, \theta_*]$.
\end{theorem}

\begin{proof}
Note that the function $\PPP+\PPP''$ is continuous and even. The theorem follows from the following statements.
\vspace{2mm}

(a) If $k < \pi/2$ then $\PPP>0$ on the interval $(-\pi/2-k,\, \pi/2+k)$ and $\PPP=0$ outside this interval. If $k \ge \pi/2$ then $\PPP>0$ everywhere, except possibly at $\theta=\pi$ (when $k=\pi/2$).
\vspace{2mm}

(b) $\PPP + \PPP'' > 0$ on the interval $(\pi/6,\, \min\{\pi/2 + k,\, \pi \}) \setminus \{ k-\pi/6 \}$.
 %$(\pi/6,\, \pi/2 + k) \setminus \{ k-\pi/6 \}$, if $k < \pi/2$, and on $(\pi/6,\, \pi) \setminus \{ k-\pi/6 \}$, otherwise.
\vspace{2mm}

(c) The function $\dfrac{1}{\cos 3\theta}\,(\PPP(\theta) + P''(\theta))$ is non-decreasing on $[0,\, \pi/6)$. Additionally, if $k \ge \pi/2$, it is strictly increasing, and if $k < \pi/2$, it is constant on a segment $[0,\, \theta_0]$ with $\theta_0>0$ sufficiently small.
\vspace{2mm}

Let us prove these statements.\vspace{2mm}

(a) Using that $f$ is supported on $[-k,\, k]$, one sees that the integrand in \eqref{S5P} is positive when $\varphi$ is in the interval $(\theta-k,\, \theta+k) \cap (-\pi/2,\, \pi/2)$ and zero when $\varphi$ lies outside the closure of this interval. For $k < \pi/2$, this interval is nonempty, and equivalently $\PPP>0$, iff $\varphi \in (-\pi/2-k,\, \pi/2+k)$. Similarly, for $k=\pi/2$,\, $\PPP>0$ on $\SSS \setminus \{ \pi \}$, and for $k> \pi/2$,\, $\PPP>0$ everywhere.
\vspace{2mm}

(b) Take $\theta \in (\pi/6,\, \min\{\pi/2 + k,\, \pi \}) \setminus \{ k-\pi/6 \}$. Using that $\cos 3\theta$ is negative on $(-\pi/2,\, -\pi/6)$ and making the change of variable $\alpha = 3\varphi - \pi/2$, we get
\begin{equation}\label{eqn1}
\PPP(\theta) + \PPP''(\theta)\ =\
-\int_{-\pi/2}^{-\pi/6} \cos 3\varphi\, f(\theta-\varphi)\, d\varphi
-\int_{-\pi/6}^{\pi/2} \cos 3\varphi\, f(\theta-\varphi)\, d\varphi
 \end{equation}
 \begin{equation}\label{eqn2}
 \ge\
-\int_{-\pi/6}^{\pi/2} \cos 3\varphi\, f(\theta-\varphi)\, d\varphi\, =\,
\dfrac{1}{3}\int_{-\pi}^{\pi} \sin\alpha\ f\Big(\theta - \frac{\pi}{6} - \dfrac{\alpha}{3} \Big)\, d\alpha
\end{equation}
 \begin{equation}\label{eqn3}
\, =\,
\dfrac{2}{3}\int_{0}^{\pi} \sin\alpha
\Big[ f\Big(\theta - \frac{\pi}{6} - \dfrac{\alpha}{3} \Big) - f\Big(\theta - \frac{\pi}{6} + \dfrac{\alpha}{3} \Big)
\Big] d\alpha.
\end{equation}
Using that $0 \le \Big|\theta - \dfrac{\pi}{6} - \dfrac{\alpha}{3}\Big| \le \pi - \Big|\dfrac{7\pi}{6} - \theta - \dfrac{\alpha}{3}\Big| \le \pi$, we have
$$
f\Big(\theta - \dfrac{\pi}{6} - \dfrac{\alpha}{3} \Big)\ =\ f\Big(\Big|\theta - \dfrac{\pi}{6} - \dfrac{\alpha}{3} \Big|\Big)\ \ge\
f\Big(\pi - \Big|\dfrac{7\pi}{6} - \theta - \dfrac{\alpha}{3}\Big|\Big) \ =\ f\Big(\theta - \dfrac{\pi}{6} + \dfrac{\alpha}{3} \Big),$$
since $f(\beta)=f(-\beta)$ and $f(\pi-\beta)=f(\pi+\beta)$. Therefore, $\PPP(\theta) + \PPP''(\theta) \ge 0$ for all $\theta$.

Now we examine carefully the cases when $\PPP(\theta) + \PPP''(\theta) > 0$. Note that $f(\theta-\varphi)>0$ if $\varphi\in(\theta-k,\theta+k)$. If either $\theta - k<-\pi/6$ or $\theta + k > 3\pi/2$, then the first integral in the right hand site of \eqref{eqn1} is negative, and therefore, the inequality in \eqref{eqn2} is strict and hence $\PPP(\theta)+\PPP''(\theta)>0$.

The opposite case is $\theta -k \ge - \pi/6$ and $\theta + k \le 3\pi/2$. Note that $\theta-k\ne-\pi/6$ by assumptions in (b). Hence $\theta -k > - \pi/6$ and $\theta + k \le 3\pi/2$. We claim that the interval
\begin{equation}\label{intervals}
3\Big( \Big|k-\theta+\dfrac{\pi}{6}\Big|, \ \min\big\{ k+\theta-\dfrac{\pi}{6},\, \frac{13\pi}{6}-\theta-k \big\}\Big) \cap [0,\, \pi]
\end{equation}
is nonempty, and the first term in the square brackets in \eqref{eqn3} is positive for the values $\alpha$ in the interval~\eqref{intervals}, while the second term is zero, and so, the integral in \eqref{eqn3} is positive.

Indeed, one easily checks that the former interval in \eqref{intervals} is empty only for $k=\pi$, which contradicts the considered case. So the former interval in \eqref{intervals} is nonempty. From the inequalities $\theta < \pi/2+k$ (condition (b)) and $k< \theta+\pi/6$ one derives that the left endpoint of the former interval lies in $[0,\, \pi)$, and therefore, the set in \eqref{intervals} is nonempty. Each $\alpha$ in the interval \eqref{intervals} satisfies $-k+\theta-\pi/6<\alpha/3<k+\theta-\pi/6\ \implies\ \theta-\pi/6-\alpha/3 \in (-k,\, k)$, hence the first term in the square brackets in \eqref{eqn3} is positive. Finally, each $\alpha$ in the interval \eqref{intervals} satisfies $k-\theta+\pi/6 <\alpha/3< 13\pi/6-\theta-k\ \implies\ \theta-\pi/6+\alpha/3 \in (k,\, 2\pi-k)$, hence the second term in the square brackets in \eqref{eqn3} is zero.

It follows that $\PPP(\theta) + \PPP''(\theta) > 0$, except possibly for a single point $\theta = k - \pi/6$.
	\vspace{2mm}

(c) By \eqref{P+P''} and \eqref{PP''} we have
$$
I(\theta) \eqdef -\int_{-\pi/2}^{\pi/2} \sin(3\theta-3\varphi)\, f(\theta-\varphi)\, d\varphi
$$ $$
=\ -\sin 3\theta \int_{-\pi/2}^{\pi/2} \cos 3\varphi\, f(\theta-\varphi)\, d\varphi\
+\ \cos 3\theta \int_{-\pi/2}^{\pi/2} \sin 3\varphi\, f(\theta-\varphi)\, d\varphi
$$ $$ =\ \sin 3\theta\, \big( \PPP(\theta) + \PPP''(\theta) \big)\
+\ \dfrac{1}{3} \cos 3\theta\, \big( \PPP(\theta) + \PPP''(\theta) \big)'.
$$
On the other hand, making the change of variable $\psi = \theta-\varphi$ and using that $f$ is even and nonnegative, for $0 < \theta < \pi/6$ we obtain (note that function $f(\psi)\sin3\psi$ is odd)
$$
I(\theta)\ =\ -\int_{-\pi/2+\theta}^{\pi/2+\theta} \sin 3\psi\, f(\psi)\, d\psi $$ $$
=	 \int_{-\pi/2}^{-\pi/2+\theta} \sin 3\psi\, f(\psi)\, d\psi
\ -\ \int_{\pi/2}^{\pi/2+\theta} \sin 3\psi\, f(\psi)\, d\psi\ \ge\ 0.
$$
If $k \ge \pi/2$, the last inequality is strict, and if $k < \pi/2$, it becomes equality for $0\le \theta \le \pi/2-k$.

So, we have
$$
\dfrac{d}{d\theta}\, \dfrac{\PPP(\theta) + \PPP''(\theta)}{\cos 3\theta}\ =\ 3\ \dfrac{\sin 3\theta\, \big( \PPP(\theta) + \PPP''(\theta) \big)\
+\ \dfrac{1}{3} \cos 3\theta\, \big( \PPP(\theta) + \PPP''(\theta) \big)'}{\cos^2 3\theta}\ $$ $$ =\ \dfrac{3I(\theta)}{\cos^2 3\theta}\ \ge \ 0.
$$
If $k \ge \pi/2$, this inequality is strict, and if $k < \pi/2$, it becomes equality for $\theta\ge 0$ sufficiently small.
\vspace{2mm}

Let us now prove the theorem.
\vspace{2mm}

(i) We have $\PPP(0)+\PPP''(0) =-K \ge 0$. Since $k \ge \pi/2$, by (c) the function $\dfrac{1}{\cos 3\theta}(\PPP(\theta)+\PPP''(\theta))$ is strictly increasing on $(0,\,\pi/6)$. According to (b), $\PPP+\PPP'' >0$ on $(\pi/6,\,\pi) \setminus \{ k-\pi/6 \}$. Hence $\PPP+\PPP'' >0$ everywhere, except possibly at $\theta = 0,\, \pm\pi/6,\, \pi,\, \pm(k-\pi/6)$. Here and in item (ii), by (a), $\PPP>0$ everywhere, except possibly at $\theta=\pi$ (when $k=\pi/2$).

(ii) We have $\PPP(0)+\PPP''(0) =-K < 0$ and, by (c), $\dfrac{1}{\cos 3\theta}(\PPP(\theta)+\PPP''(\theta))$ is strictly increasing on $(0,\,\pi/6)$. According to (b), $\PPP+\PPP'' >0$ on $(\pi/6,\,\pi) \setminus \{ k-\pi/6 \}$. It follows that $\PPP+\PPP'' \le 0$ on a certain interval $[0,\, \theta_*]$,\, $0<\theta_*\le\pi/6$, and $\PPP+\PPP'' > 0$ on $(\theta_*,\, \pi) \setminus \{ \pi/6,\, k-\pi/6 \}$, and the same inequalities are true on the symmetric intervals $[-\theta_*,\, 0]$ and $(-\pi,\, -\theta_*) \setminus \{ -\pi/6,\, -k+\pi/6 \}$. Thus, $\PPP+\PPP'' \le 0$ on $[-\theta_*,\, \theta_*]$ and $\PPP+\PPP'' >0$ on $\hat\SSS \setminus [-\theta_*,\, \theta_*].$

(iii) We have $\PPP(0)+\PPP''(0) =-K > 0$. Since $k < \pi/2$, by (c), $\dfrac{1}{\cos 3\theta}(\PPP(\theta)+\PPP''(\theta))$ is non-decreasing on $(0,\,\pi/6)$. According to (b), $\PPP+\PPP'' >0$ on $(\pi/6,\,\pi/2+k)\setminus \{ k-\pi/6 \}$.  It follows that $\PPP+\PPP'' >0$ on $(-\pi/2-k,\, \pi/2+k) \setminus \{ \pm\pi/6,\, \pm(k-\pi/6) \}$. Here and in item (iv), by (a), $\PPP>0$ on $(-\pi/2-k,\, \pi/2+k)$ and $\PPP=0$ on $\SSS \setminus (-\pi/2-k,\, \pi/2+k)$.

(iv) We have $\PPP(0) + \PPP''(0) = -K\le 0$. Since $k < \pi/2$, by (c), $\PPP(\theta) + \PPP''(\theta) = -K\cos 3\theta \le 0$ on a certain segment $[0,\,\theta_0]$, $0<\theta_0<\pi/6$. By (c) and (b), $\PPP+\PPP'' \le 0$ on $[0,\, \theta_*]$, where $\theta_0 \le \theta_* \le \pi/6$,\, $\PPP+\PPP'' >0$ on $(\theta_*,\, \pi/2+k)\setminus \{ \pi/6,\, k-\pi/6 \}$ (and the same inequalities hold true on the symmetric segments). Thus, the last item of Theorem \ref{t_cond} is proved.
\end{proof}

The following important statement is a simple consequence of Theorem \ref{t_cond}.

\begin{corollary}\label{coro}
The boundary of an optimal body has a singular point at the top, if condition C1 is satisfied, and does not have otherwise, and it has a singular point at the bottom if condition C2 is satisfied, and does not have otherwise. The rest of the boundary is $C^1$-regular.
\end{corollary}

\begin{proof}
	
The boundary of $\Omega$ can be parameterized by
$$
r(\theta) = \frac{1}{\PPP(\theta)}\, e_{\theta}, \qquad \theta\in \big(-\min\{k+\pi/2, \pi\},\ \min\{k+\pi/2, \pi \} \big).
$$
The remaining values of $\theta$ correspond to zero values of $\PPP$. It is straightforward to calculate
$$
r' \times r''\ =\ \frac{\PPP +\PPP''}{\PPP^3};
$$
that is, the curve goes counterclockwise about the origin, turns left when $\PPP+ \PPP'' >0$, turns right when $\PPP+ \PPP''<0$, and goes straight when $\PPP+ \PPP''=0$.

(i) If both conditions C1 and C2 are not satisfied then $\PPP+\PPP''>0$, except possibly for several (up to 6) isolated points, hence $\Omega$ is strictly convex; see Fig.~\ref{fig-cor}\,(i). If $k > \pi/2$ then $\Omega$ is bounded, and if $k=\pi/2$ then it contains a single ray corresponding to $\theta=\pi$. It follows that the boundary of the corresponding optimal body is regular.

(ii) If only condition C1 is satisfied then the curve initially turns left, then on $[-\theta_*, \theta_*]$ it turns right or goes straight, and then again turns left. It follows that $\Omega$ is not convex. The boundary of $\tilde\Omega$ contains one line interval in its front part (shown as a dashed line in Fig.~\ref{fig-cor}\,(ii)) with the angular width greater than or equal to $2\theta_*.$  Again, if $k > \pi/2$ then $\Omega$ is bounded, and if $k=\pi/2$ then it contains a single ray corresponding to $\theta=\pi$. The corresponding optimal body has a unique singular point at the top of its boundary, and the angle at the top is smaller than or equal to  $\pi-2\theta_*$.

(iii) If only condition C2 is satisfied then $r$ is defined on $(-\pi/2-k,\, \pi/2+k)$ and $\PPP+\PPP''>0$ on its domain, except for at most 6 points. It follows that $\Omega$ is convex and unbounded (see Fig.~\ref{fig-cor}\,(iii)), moreover it contains rays, with the vertex at the origin, corresponding to the angles in $[\pi/2+k,\, 3\pi/2-k]$. Hence the optimal body has a unique singularity at the bottom of its boundary, and the angle at that point equals $2k.$

(iv) If both C1 and C2 are satisfied, then again, the curve determining the boundary turns left, then turns right or goes straight, and then again turns left. $\Omega$ is nonconvex and unbounded and, moreover, contains rays corresponding to the angles in $[\pi/2+k,\, 3\pi/2-k]$; see Fig.~\ref{fig-cor}\,(iv). The boundary of $\tilde\Omega$ contains one line interval shown dashed in the figure. The corresponding optimal body has two singularities, at the top and at the bottom, the angle at the top is.

\end{proof}

\begin{remark}
Note that the boundary of an optimal body is $C^1$ except for at most 2 singular points, and is $C^2$, except for up to 6 additional points, where it is $C^1\setminus C^2$ regular. Indeed, the boundary of the polar set $\Omega^\circ$ is $C^2$ on a segment iff curvature of $\partial(\conv\Omega)$ is positive on the corresponding segment, and the latter is $\PPP^2(\PPP+\PPP'') (\PPP^2+\PPP'^2)^{-3/2}$ (i.e.\ these additional points are determined by condition $\PPP+\PPP''=0$).
\end{remark}

{\bf 5.2.} \ Now consider the application of the theory to the special case when a body makes uniform rotations with the amplitude $a$,\, $0 < a \le \pi$. We calculate the functions $\PPP=\PPP_a$ and $Q=Q_a$.  Here we use the formula
$$
P(\theta)\ =\ \int\limits_{[\theta-a,\theta+a]\cap[-\pi/2,\pi/2]} \cos^3 \varphi\, d\varphi,
$$
which coincides with \eqref{Punif} up to the factor $1/(2a)$.
Since $\PPP$ is even, it suffices to calculate its values for $\theta \in [0,\, \pi].$

Let us consider separately 5 regions on the square $0 \le a \le \pi$,\, $0 \le \theta \le \pi$ in the $(a, \theta)$-plane; see Fig.~\ref{fig:square}.

\begin{figure}[ht]
	\begin{center}
		\includegraphics[width=0.5\textwidth]{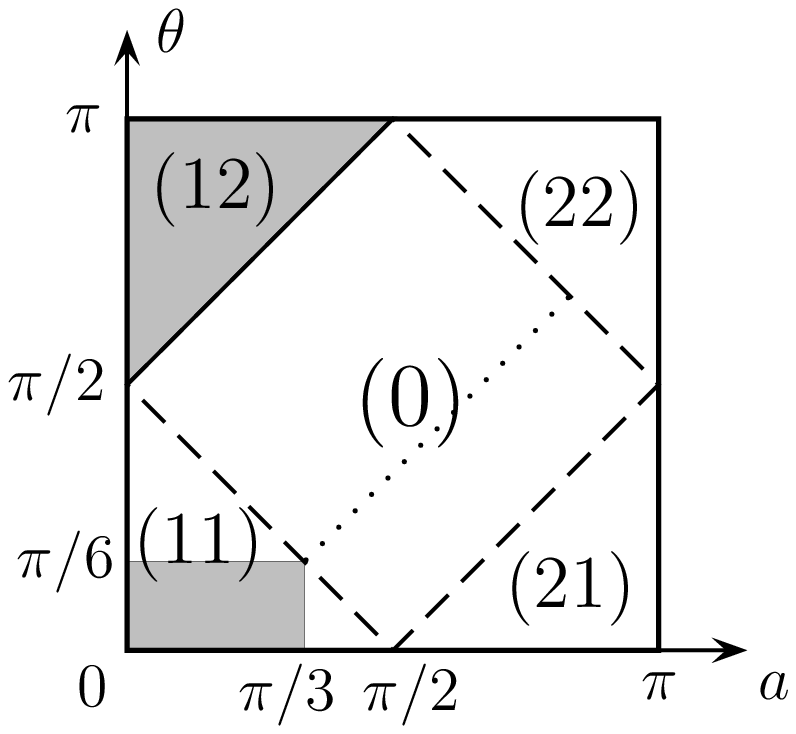}
\caption{Five regions. The rectangular shaded region corresponds to $\PPP+\PPP''\le0$. The triangular shaded region corresponds to $\PPP=0.$ The dotted diagonal line corresponds to $\PPP+\PPP''=0$.}
\label{fig:square}
	\end{center}
\end{figure}

(11)~~$0 \le \theta \le \pi/2 - a$. Here we have $[\theta-a,\, \theta+a]\cap[-\pi/2,\, \pi/2]=[\theta-a,\, \theta+a]$,
$$
\PPP(\theta)\ = \ \frac32\ \sin a\cos\theta + \frac16\sin(3a)\cos(3\theta)
$$ and $$
Q(\xi, \eta)\ =\ \mu_\Omega(\eta,-\xi)\ = \ 2\sin a\ \frac{\eta\, \big[  \xi^2\sin^2 a + \eta^2(1 - \frac{1}{3}\sin^2 a) \big]}{\xi^2 + \eta^2}.
$$
Further,
$$\PPP(\theta) + \PPP''(\theta)\ =\ -\frac{4}{3}\, \sin(3a)\cos(3\theta),$$
hence $$
\PPP+\PPP''\le 0 \qquad \Longleftrightarrow \qquad 0\le a\le\pi/3 \ \ \text{and} \ \ 0\le\theta\le\pi/6.
$$
If $a=\pi/3$,\, $\PPP+\PPP''=0$ on $[0,\, \pi/6]$ and $\PPP+\PPP''>0$ on $(\pi/6,\, \pi/2-a]$. If $0<a<\pi/3$,\, $\PPP+\PPP''<0$ on $[0,\, \pi/6)$ and $\PPP+\PPP''>0$ on $(\pi/6,\, \pi/2-a]$. The corresponding set $\Omega$ looks something like the ones shown in Fig.~\ref{fig:a} for the values $0<a<\pi/3$ and $a=\pi/3$.
%%%%%%%%%%%%%%
\begin{figure}[ht]
	\begin{center}
		\includegraphics[width=0.5\textwidth]{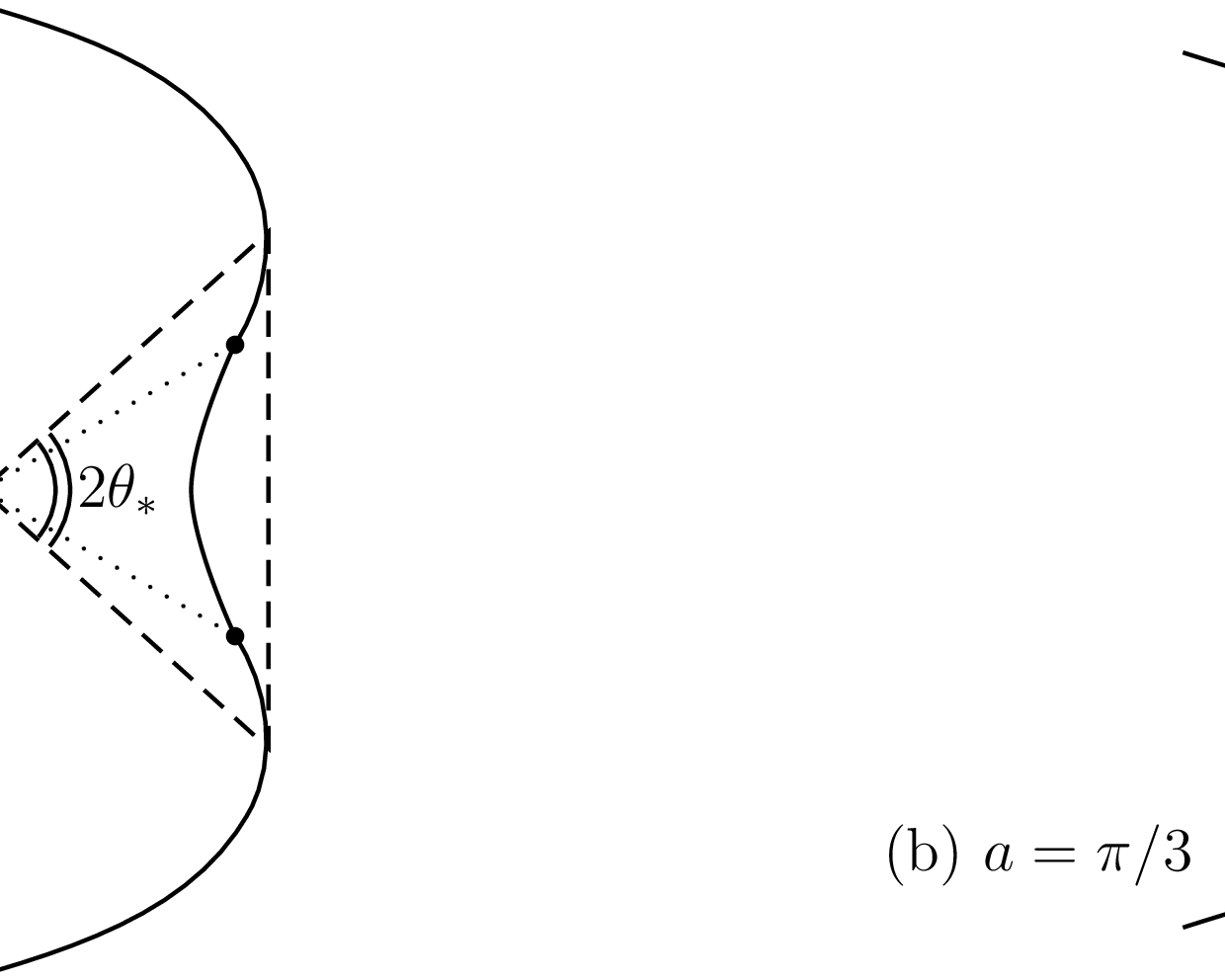}
\caption{The set $\Omega$ in the case of uniform rotation looks similar to those shown in Fig.~(a) for $0<a<\pi/6$ and in Fig.~(b) for $a=\pi/3$. The dot is at the origin.}
\label{fig:a}
	\end{center}
\end{figure}
%%%%%%%%%%%%%%
Hence for $a=\pi/3$, the tangent cone at the front part of the body's boundary is $2\pi/3$, and for $0<a<\pi/3$, the tangent cone is $\pi-2\theta_*<2\pi/3$, where $\theta_*$ (obviously depending on $a$) ia as in Fig.~\ref{fig:a}\,(a). If $\pi/6<a<\pi/2$, there is no singularity at the front.

(12)~~$\pi/2 + a \le \theta \le \pi$. Here we have $[\theta-a,\, \theta+a]\cap[-\pi/2,\, \pi/2]=\emptyset$, hence
$$
\PPP(\theta)\ =\ 0 \quad \text{and} \quad Q(\xi, \eta)\ = \ 0.
$$

(21)~~$0 \le \theta \le a - \pi/2$. Here we have $[\theta-a,\, \theta+a]\cap[-\pi/2,\, \pi/2]=[-\pi/2,\, \pi/2]$, hence
$$
\PPP(\theta)\ =\ \frac{4}{3} \quad \text{and} \quad Q(\xi, \eta)\ = \ \frac{4}{3}\, \sqrt{\xi^2 + \eta^2}.
$$

(22)~~$3\pi/2 - a \le \theta \le \pi$. Here we have $[\theta-a,\, \theta+a]\cap[-\pi/2,\, \pi/2]=[\theta-a,\, \theta+a]\setminus[\pi/2,\, 3\pi/2]$,
$$
\PPP(\theta)\ =\ \frac{4}{3}\ +\ \frac32\ \sin a\cos\theta + \frac16\sin(3a)\cos(3\theta)
$$ and $$
Q(\xi, \eta)\ = \ \frac{4}{3}\, \sqrt{\xi^2 + \eta^2}\ +\ 2\sin a\ \frac{\eta\, \big[ \xi^2\sin^2 a + \eta^2 (1 - \frac{1}{3}\sin^2 a)\big]}{\xi^2 + \eta^2}.
$$
$$\PPP(\theta)+\PPP''(\theta) = \frac{4}{3}\,\big[1-\sin(3a)\cos(3\theta)\big],$$
hence $\PPP+\PPP''>0$ everywhere, except for the points $a=\pi/2,\, \theta=\pi$ and $a=5\pi/6,\, \theta=2\pi/3.$

(0)~~$|\theta-\pi/2|+|a-\pi/2|\le \pi/2$. Here we have $[\theta-a,\, \theta+a]\cap[-\pi/2,\, \pi/2]=[\theta-a,\, \pi/2]$,
$$
\PPP(\theta)\ =\ \frac{2}{3}\, -\, \sin(\theta-a)\, +\, \frac{1}{3} \sin^3(\theta-a)
$$ and $$
Q(\xi, \eta)\ = \ \frac{2}{3}\, \sqrt{\xi^2 + \eta^2}\, -\, \cos a\, |\xi| + \sin a\, \eta\, -\, \frac{1}{3}\, \frac{(-\cos a\, |\xi| + \sin a\, \eta)^3}{\xi^2 + \eta^2}.
$$
$$\PPP(\theta)+\PPP''(\theta)= \frac{2}{3}\,\big[1+\sin(3\theta-3a)\big],$$
hence $\PPP+\PPP''>0$ everywhere, except for the lines $\theta-a=\pi/2$ and $\theta-a=-\pi/6.$
\vspace{2mm}

Summarizing, the front point of the body's boundary is singular iff $0<a\le\pi/3$ (see case (11)), and the bottom point is singular iff $0<a<\pi/2$ (see case (12)). The tangent cone at the front is smaller than $2\pi/3$ when $a<\pi/3$ and reaches this value when $a=\pi/3$. The tangent cone at the back part equals $2a.$

Consider several particular cases.
\vspace{2mm}

(i) In the limit of pure translation $a \to 0$ we have the following asymptotics of $Q=Q_a$,
$$
\frac{1}{2a}\, Q_a(\xi,\eta)\ \to \ \left\{
\begin{array}{ll}
	\frac{\eta^3}{\xi^2 + \eta^2}, & \text{if } \ \eta \ge 0\,;\\
	0, & \text{if } \ \eta < 0\,.
\end{array}
\right.
$$
The length of the optimal body goes to infinity, and the width goes to zero. An approximate shape of the optimal body for $a$ small is shown in Fig.~\ref{fig:Shapes}\,(i).
\vspace{2mm}
 %\rput(7,-1.3){\psarc(0,-20.35){21}{80}{100}\psarc(0,21){21}{-100}{-80}}\vspace{20mm}

(ii) $a = \pi/3$.
$$ Q(\xi,\eta)\ =\
\left\{
\begin{array}{ll}
	\frac{3\sqrt{3}}{4}\, \eta, & \text{if } \ \eta \ge \sqrt{3}\, |\xi|\,;\\
	\frac{2}{3}\, \sqrt{\xi^2 + \eta^2} - \frac{1}{3}\, \frac{|\xi|^3}{\xi^2 + \eta^2} + \frac{3\sqrt{3}\, \eta - |\xi|}{8}, & \text{if } \ -\sqrt{3}\, |\xi| \le \eta \le \sqrt{3}\, |\xi|\,;\\
	0, & \text{if } \ \eta \le -\sqrt{3}\, |\xi|\,.
\end{array}
\right.
$$
optimal body is bounded by the union of two curves
%$$\left( \frac{3\sqrt{3}}{8} + \frac{2}{3}\, \sin\varphi + \frac{2}{3}\, \cos^3 \varphi \sin\varphi, \ \ \pm\Big[ -\frac{1}{8} + \frac{2}{3}\, \cos\varphi - \frac{1}{3}\, \cos^4 \varphi - \cos^2 \varphi \sin^2 \varphi\Big], \right), $$$-\pi/3 \le \varphi \le \pi/3$; see the figure below.
$$
\left( \frac{3\sqrt{3}}{8} +
\frac{2}{3}\,\cos\theta\,\big(1 + \sin^3 \theta \big),
\ \ \pm\Big[ -\frac{1}{8}+
\frac{2}{3}\,\sin\theta\,\big(1 + \sin^3 \theta \big)
-\sin^2 \theta\, \Big] \right),
$$
$\pi/6 \le \theta \le 5\pi/6$; see Fig.~\ref{fig:Shapes}\,(ii).
\vspace{2mm}
%\rput(7.33,-0.8){\scalebox{2.98}{		\pscurve[linewidth=0.4pt](0.6495,0)(0.621,0.02845)(0.417,0.1383)(0,0.208)(-0.417,0.1383)(-0.621,0.02845)(-0.6495,0)\pscurve[linewidth=0.4pt](0.6495,0)(0.621,-0.02845)(0.417,-0.1383)(0,-0.208)(-0.417,-0.1383)(-0.621,-0.02845)(-0.6495,0)
	
		%\pscurve[linewidth=0.4pt](0,0.6495)(-0.02845,0.621)(-0.1383,0.417)(-0.208,0)(-0.1383,-0.417)(-0.02845,-0.621)(0,-0.6495)}}
%\rput(7.45,-0.8){\scalebox{1}{ \pscurve[linewidth=0.6pt,linecolor=blue] (1.94856,0)(1.94467,0.005426)(1.91421,0.03921)(1.6495,0.2321)(0,0.625)(-1.6495,0.2321)(-1.91421,0.03921)(-1.94467,0.005426)(-1.94856,0)
%\pscurve[linewidth=0.6pt] (1.94856,0)(1.94467,-0.005426)(1.91421,-0.03921)(1.6495,-0.2321)(0,-0.625)(-1.6495,-0.2321)(-1.91421,-0.03921)(-1.94467,-0.005426)(-1.94856,0)}}

(iii) $a = \pi/2$.
$$
Q(\xi,\eta)\ =\ \frac{2}{3}\, \sqrt{\xi^2 + \eta^2}\, +\, \eta\, -\, \frac{1}{3}\  \frac{\eta^3}{\xi^2 + \eta^2}.
$$
The optimal body is bounded by the curve
%$$\big( 2\cos\varphi + 2\cos\varphi \sin^3 \varphi, \ 2\sin\varphi + 3 - 3\sin^2 \varphi \cos^2 \varphi - \sin^4 \varphi \big), \quad 0 \le \varphi \le 2\pi;$$
%$$\Big( \frac{2}{3}\,\sin\varphi(1+ \sin^3 \varphi) + \cos^2\varphi, \, \ \frac{2}{3}\,\cos\varphi(1 + \sin^3 \varphi) \Big), \quad 0 \le \varphi \le 2\pi;$$ see the figure below.
$$
\left(\frac{2}{3}\,\cos\theta\,\big(1 + \cos^3\theta \big)
+\sin^2\theta, \ \
\frac{2}{3}\,\sin\theta\,\big(1 + \cos^3\theta \big)\right), \quad 0 \le \theta \le 2\pi;
$$
see Fig.~\ref{fig:Shapes}\,(iii).

Note that this body is a shape of constant width\footnote{This is an observation by Valentin Shehtman.}. This is a direct consequence of the fact that the function $\PPP(\theta)= \PPP_{\pi/2}(\theta)\ =\ \frac{2}{3}+ \cos\theta- \frac{1}{3} \cos^3\theta$ satisfies $\PPP(\theta)+\PPP(\theta+\pi)=\const.$
\vspace{2mm}
%\begin{figure}[h]\begin{picture}(0,120)\rput(5.8,2.4){\pscurve(3,2)(3.375,1.949)(3.607,1.65)(4,0)(3.607,-1.65)(3.375,-1.949)(3,-2)(1.375,-1.516)(0.143,-0.35)(0,0)(0.143,0.35)(1.375,1.516)(3,2)}\end{picture}\end{figure}

(iv) $a = 2\pi/3$.
$$ Q(\xi,\eta)\ =\
\left\{
\begin{array}{ll}
	4/3\, \sqrt{\xi^2 + \eta^2}, & \text{if } \ \eta \ge \sqrt{3}\, |\xi|\,;\\
	\frac{2}{3}\, \sqrt{\xi^2 + \eta^2} + \frac{1}{3}\, \frac{|\xi|^3}{\xi^2 + \eta^2} + \frac{3\sqrt{3}\, \eta + |\xi|}{8}, & \text{if } \ -\sqrt{3}\, |\xi| \le \eta \le \sqrt{3}\, |\xi|\,;\\
	4/3\, \sqrt{\xi^2 + \eta^2} + \frac{3\sqrt{3}}{4}\, \eta, & \text{if } \ \eta \le -\sqrt{3}\, |\xi|\,.
\end{array}
\right.
$$
The optimal body
%polar set of the set $\{ Q \le 1 \}$
is bounded by the union of two curves
%$$\left( \frac{3\sqrt{3}}{8}+\frac{2}{3}\,\sin\varphi - \frac{2}{3}\, \cos^3 \varphi \sin\varphi, \ \ \pm\Big[ \frac{1}{8} +\frac{2}{3}\, \cos\varphi + \frac{1}{3}\, \cos^4 \varphi + \cos^2 \varphi \sin^2 \varphi\Big] \right)$$
$$
\left( \frac{3\sqrt{3}}{8}+
\frac{2}{3}\,\cos\theta\,\big( 1-\sin^3\theta\big), \ \
\pm\Big[ \frac{1}{8} +\frac{2}{3}\,\sin\theta\,
\big( 1-\sin^3\theta\big)+ \sin^2\theta \Big]\right),
$$
$\pi/6 \le \theta \le 5\pi/6$ and by two arcs of circumferences of radius $4/3$; the first one with the center at $(0,0)$ contained in the cone $x \le \sqrt 3\,|y|$, and the second one with the center at $(3\sqrt{3}/4,0)$ contained in the cone $x \le 3\sqrt{3}/4-\sqrt 3\,|y|$. See Fig.~\ref{fig:Shapes}\,(iv).
\vspace{2mm}

(v) $a = \pi$.
$$
Q(\xi,\eta)\ =\ \frac{4}{3}\, \sqrt{\xi^2 + \eta^2}.
$$
In this case the optimal body is a circle; see Fig.~\ref{fig:Shapes}\,(v).
%\rput(5,0.15){\pscircle(0,0){1.5}}

\begin{figure}[ht]
	\begin{center}
		\includegraphics[width=0.5\textwidth]{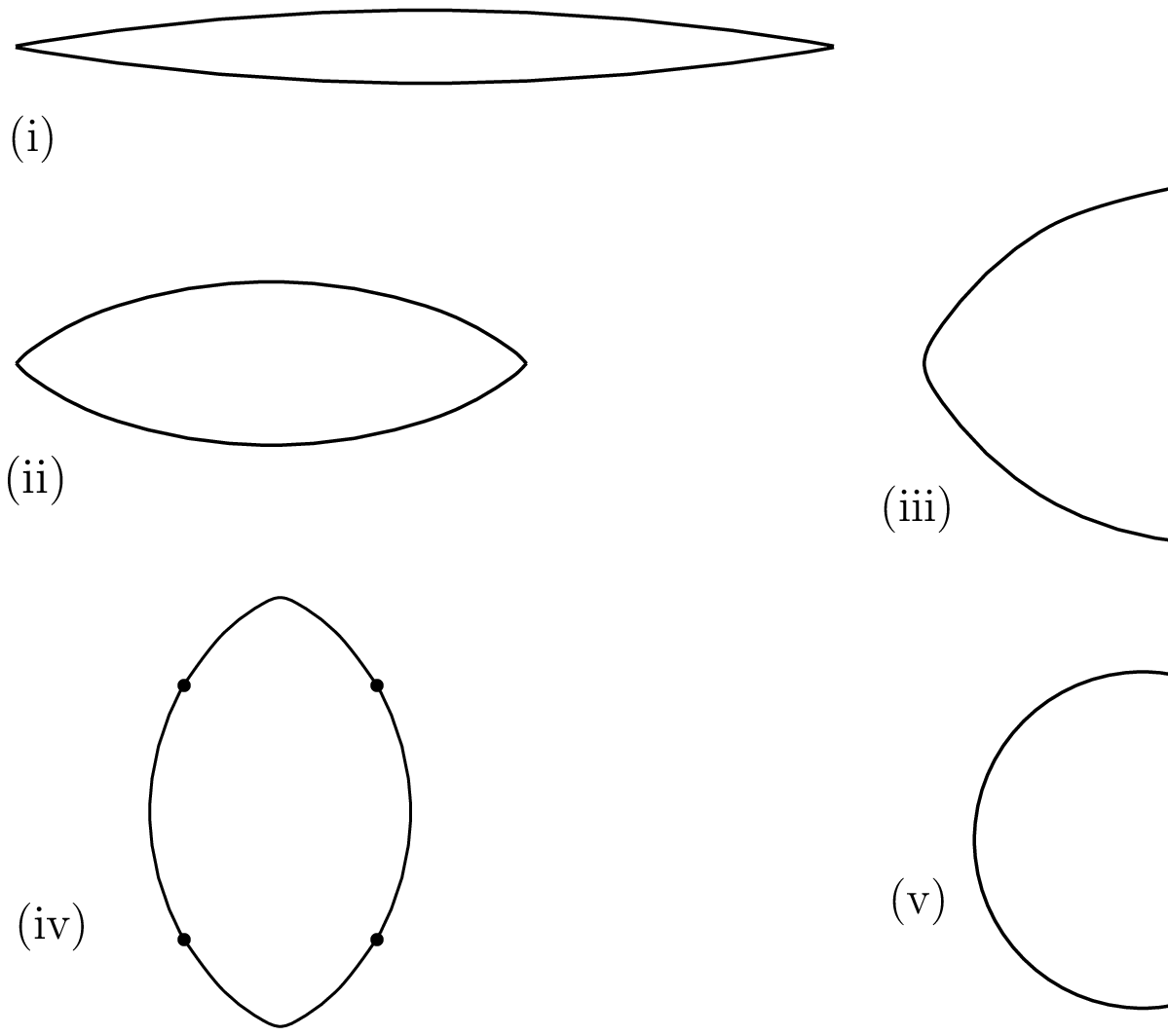}
\caption{The optimal shapes in the case of uniform rotation for (i) $a \ll 1$; (ii) $a=\pi/3$; (iii) $a=\pi/2$; (iv) $a=2\pi/3$; (v) $a=\pi$.}
\label{fig:Shapes}
	\end{center}
\end{figure}

\section{Conclusion}

Here we consider planar convex bodies moving in a rarefied medium on the plane and making slow oscillations. The particles of the medium are originally at rest and reflect elastically when colliding with the body. We are looking for a body, with a prescribed area, that has minimal average resistance.
We have seen that the optimal shape is entirely determined by the density function $f$ of the distribution of rotation angles, while the value of area is responsible merely for the size rather than shape of the optimal body.

The algorithm of solving the problem is the following. Knowing $f$, we then calculate the function $\PPP$, which determines the time-averaged projection of the force induced by the pressure on the body's boundary on the direction of motion. Then we determine the planar set $\Omega$, which is defined by $r \PPP(\theta) \le 1$ in the polar coordinates $r,\, \theta$. Next we take the convex hull $\tilde\Omega = \rm{conv}\,\Omega$ and find the polar set $\Omega^{\circ}$ to $\tilde\Omega$. The optimal body is similar to the set $\Omega^{\circ}$ with the coefficient of similarity being proportional to the square root of the area. The solution scheme can be briefly summarized as follows:
$$
f\, \ \mapsto\, \ \PPP\, \ \mapsto\, \ \Omega\ \, \mapsto\ \, \text{optimal } B \approx
(\conv\,\Omega)^\circ.
$$

Basically, this problem is a generalization of the isoperimetric problem for Finsler metrics on the plane.

Further, we prove that under mild assumptions on $f$, an optimal body can have singularities only at the front and back points of its boundary, and derive necessary and sufficient conditions for each of these singularities. Roughly speaking, a singularity takes place when the oscillation is, in a sense, small, and the smallness is understood differently in the top and in the bottom. The rest of the body's boundary is regular.

We apply the obtained results to the particular case of uniform oscillations with amplitude $a$, that is, when $f = \chi_{[-a,a]}$, $0<a\le\pi$.
It is interesting to follow the behavior of singularities at the front and rear points of the optimal body as $a$ changes from 0 to $\pi$. For $0<a\le\pi/3$ there are both singularities, for $\pi/3<a<\pi/2$ only the rear singularity survives, and for $\pi/2\le a\le\pi$ the boundary of the body is $C^1$-smooth. The tangent cone at the front point is $<2\pi/3$ for $0<a<\pi/3$ (we suppose that it increases monotonically from 0 to $\pi/3$) and is $2\pi/3$ for $a=\pi/3$, and the singularity abruptly disappears as $a$ increases further. The tangent cone at the rear point is $2a$; that is, it increases monotonically from 0 to $\pi$ as $a$ grows from 0 to $\pi/2.$
%The front singularity of the optimal body appears when $a\le\pi/3$, and the backward singularity appears when $a<\pi/2$.
Several optimal shapes are constructed explicitly.

\section*{Acknowledgements}

The work of AP was partly supported by CIDMA (https://ror.org/05pm2mw36) under the Portuguese Foundation for Science and Technology (FCT, https://ror.org/00snfqn58), Grants UID/04106/2025 (https://doi.org/10.54499/UID/04106/2025) and UID/PRR/04106/2025.


\begin{thebibliography}{99}

\bibitem{AP}
A. Akopyan and A. Plakhov. {\it Minimal resistance of curves under the single impact assumption.} SIAM J. Math. Anal. {\bf 47}, 2754-2769 (2015). %DOI. 10.1137/140993843

\bibitem{Lokut2} Ardentov, A. A., Lokutsievskiy, L. V., Sachkov, Yu. L. Extremals for a series of sub-Finsler problems with 2-dimensional control via convex trigonometry // ESAIM: Control, Optimisation and Calculus of Variations. — 2021. — Vol. 27. — Article 32 (52 pp.).
%DOI: 10.1051/cocv/2021024 (arXiv: 2004.10194).

\bibitem{Aleksandrov}
A.\,D. Aleksandrov. {\it On the theory of mixed volumes of convex bodies. III. Extending two theorems of Minkowski on convex polytopes to arbitrary convex bodies}.\, Mat. Sb. {\bf 3 (45)}:1, 27-46 (in Russian) (1938).
     %%\bibitem{0-resist}A. Aleksenko and A. Plakhov. {\it Bodies of zero resistance and bodies invisible in one direction}. Nonlinearity {\bf 22}, 1247-1258 (2009).

     %\bibitem{Alex}A. D. Alexandrov.\, {\it Selected Works. Part I: Selected Scientific Papers}, { Chapter V, \S 3.}\, Ed. by Yu.~G. Reshetnyak and S.~S. Kutateladze. Gordon and Breach Publishers (1996). %{\it To the theory of mixed volumes for convex bodies. Part I: Extension of certain concepts of the theory of convex bodies}.

      %\bibitem{BelloniKawohl}M. Belloni and B. Kawohl. \textit{A paper of Legendre revisited}. Forum Math. {\bf 9}, 655-668 (1997).

     %\bibitem{BW prescribed volume}M. Belloni and A. Wagner. {\it Newton’s problem of minimal resistance in the class of bodies with prescribed volume}.J. Convex Anal. {\bf 10}, 491–500 (2003).
\bibitem{Berestovskii} Berestovskii, V. N. Geodesics of nonholonomic left-invariant intrinsic metrics on the Heisenberg group and isoperimetric curves on the Minkowski plane // Siberian Mathematical Journal. — 1994. — Vol. 35, no. 1. — P. 1–8

\bibitem{BrFK}
F. Brock, V. Ferone and B. Kawohl.\, \textit{A symmetry problem in the calculus of variations}.\, Calc. Var. {\bf 4}, 593-599 (1996).

\bibitem{BFK}
G. Buttazzo, V. Ferone, B. Kawohl.\,
\textit{Minimum problems over sets of concave functions and related questions}.\, Math. Nachr. {\bf 173}, 71--89 (1995).

\bibitem{BK}
G. Buttazzo, B. Kawohl.\, \textit{On Newton's problem of minimal resistance}.\, Math. Intell. {\bf 15}, 7--12 (1993).

     %\bibitem{BG97} G. Buttazzo, P. Guasoni.\, {\it Shape optimization problems over classes of convex domains}.\, J. Convex Anal. {\bf 4}, No.2, 343-351 (1997).

\bibitem{Busemann} Busemann, H. The Isoperimetric Problem in the Minkowski Plane // American Journal of Mathematics. — 1947. — Vol. 69, no. 4. — P. 763–771.
%DOI: 10.2307/2371807

\bibitem{CL1}
M. Comte, T. Lachand-Robert.\, \textit{Newton's problem of the body of minimal resistance under a single-impact assumption}.\,
Calc. Var. Partial Differ. Equ. {\bf 12}, 173-211 (2001).

\bibitem{CL2}
M. Comte, T. Lachand-Robert.\, \textit{Existence of minimizers for Newton's problem of the body of minimal resistance under a single-impact assumption}.\, J. Anal. Math. {\bf 83}, 313-335 (2001).

\bibitem{CL3}
M. Comte and T. Lachand-Robert. {\it Functions and domains having minimal resistance under a single-impact assumption}.\, SIAM J. Math. Anal. {\bf 34}, 101-120 (2002).

\bibitem{DP}
A. Davydov, A Plakhov. {\it Aerodynamic stabilization of rod motion in a flat rarefied medium}. Proceedings of the Steklov Institute of Mathematics (in press).

  %\bibitem{GK} A. Glutsyuk and Y. Kudryashov. {\it  No planar billiard possesses an open set of quadrilateral trajectories}. J. Modern Dynam. {\bf 6}, 287-326 (2012).
  % doi: 10.3934/jmd.2012.6.287

 %\bibitem{Iv} V. Ya. Ivrii. {\it The second term of the spectral asymptotics for a Laplace-Beltrami operator on manifolds with boundary}. Funct. Anal. Appl. {\bf 14}, 98–106 (1980)

\bibitem{Guasoni}
P. Guasoni. {\it Problemi di ottimizzazione di forma su classi di insiemi convessi.} Tesi di Laurea. Università degli Studi di Pisa (in Italian) (1995).


\bibitem{Kaw}
B. Kawohl.  {\it Some nonconvex shape optimization problems}. In: Cellina, A., Ornelas, A. (eds) Optimal Shape Design. Lecture Notes in Mathematics, vol 1740. Springer, Berlin, Heidelberg (2000).  %https://doi.org/10.1007/BFb0106741

\bibitem{Kryzh}
S. Kryzhevich. {\it Motion of a rough disc in Newtonian aerodynamics.} In: Plakhov, A., Tchemisova, T., Freitas, A. (eds) Optimization in the Natural Sciences. EmC-ONS 2014. Communications in Computer and Information Science, vol 499. Springer. %, Cham. https://doi.org/10.1007/978-3-319-20352-2_1


\bibitem{LO}
T. Lachand-Robert and E. Oudet.\, \textit{Minimizing within convex bodies using a convex hull method}.\, SIAM J. Optim. {\bf 16}, 368-379 (2006).

\bibitem{LP1}
T. Lachand-Robert, M.~A. Peletier.\,
\textit{Newton's problem of the body of minimal resistance in the class of convex developable functions}.\, Math. Nachr. {\bf 226}, 153-176 (2001).

       %\bibitem{LP2} T. Lachand-Robert, M.~A. Peletier.\, \textit{An example of non-convex minimization and  an application to Newton's problem of the body of least resistance}.\, Ann. Inst. H. Poincar\'e, Anal. Non Lin. {\bf 18}, 179-198 (2001).

\bibitem{Lokut1}  L. V. Lokutsievskiy, {\it Convex trigonometry with applications to sub-Finsler geometry}. Sb. Math., 210:8, 1179–1205 (2019).
%DOI: 10.1016/S0031-3203(02)00090-0

\bibitem{LZ}
L.\,V. Lokutsievskiy and M.\,I. Zelikin. {\it
The analytical solution of Newton’s aerodynamic problem in the class of bodies with vertical plane of symmetry and developable side boundary.}
ESAIM: COCV {\bf 26}, 15, 36 pages (2020).

\bibitem{LWZ}
L. Lokutsievskiy, G. Wachsmuth, and M. Zelikin. {\it
Non-optimality of conical parts for Newton’s problem of minimal resistance in the class of convex bodies and the limiting case of infinite height.} Calc. Var. {\bf 61}, 31 (2022).

\bibitem{N}
I. Newton. {\it Philosophiae naturalis principia mathematica}. (London: Streater) 1687.

\bibitem{DAN2003}
A Plakhov. {\it Newton's problem of a body of minimal aerodynamic resistance}. Doklady Math. {\bf 390}, 314-317 (2003).

\bibitem{matsb2004}
A Plakhov. {\it Newton's problem of the body of minimum mean resistance}. Sbornik: Math. {\bf 195}, 1017-1037 (2004).
%\\ doi: 10.1070/SM2004v195n07ABEH000836

\bibitem{ARMA}
A. Plakhov. {\it Billiards and two-dimensional problems of optimal resistance}. Arch. Ration. Mech. Anal. {\bf 194}, 349-382 (2009). %DOI

\bibitem{rough2D}
A. Plakhov. {\it Billiard scattering on rough sets: Two-dimensional case}. SIAM J. Math. Anal. {\bf 40}, 2155-2178 (2009). %DOI. 10.1137/070709700.

\bibitem{RMS2009}
A Plakhov. {\it Scattering in billiards and problems of Newtonian aerodynamics}. Russ. Math. Surv. {\bf 64}, 873-938 (2009).
%\\ doi: 10.1070/RM2009v064n05ABEH004642

 %\bibitem{optimal roughening} A Plakhov. {\it Optimal roughening of convex bodies}. Canad. J. Math. {\bf 64}, 1058-1074 (2012). %doi: 10.4153/CJM-2011-070-9

\bibitem{bookP}
A. Plakhov. {\it Exterior billiards. Systems with impacts outside bounded domains}. Springer, New York, 2012. xiv+284 pp. ISBN: 978-1-4614-4480-0

\bibitem{SIC}
A. Plakhov. {\it Newton’s problem of minimal resistance under the single-impact assumption}. Nonlinearity {\bf 29}, 465-488 (2016).

\bibitem{SIREV}
A. Plakhov. {\it Problems of minimal resistance and the Kakeya problem.} SIAM Review {\bf 57}, 421-434 (2015). %DOI. 10.1137/15M1012931

%\bibitem{invisNpoints}  A. Plakhov. {\it Plane sets invisible in finitely many directions}. Nonlinearity {\bf 31}, 3914-3938 (2018). %https://doi.org/10.1088/1361-6544/aac63b

\bibitem{MMO}
A. Plakhov. {\it On generalized Newton’s aerodynamic problem}. Trans. Moscow Math. Soc. {\bf 82}, 217-226 (2021).

\bibitem{P-sing}
A. Plakhov. {\it A solution to Newton's least resistance problem is uniquely defined by its singular set}. Calc. Var. {\bf 61}, 189 (2022). %https://doi.org/10.1007/s00526-022-02300-w

\bibitem{P-boundary}
A. Plakhov. {\it A note on Newton's problem of minimal resistance for convex bodies}. Calc. Var. {\bf 59},167 (2020).
%DOI: https://doi.org/10.1007/s00526-020-01833-2

\bibitem{sbmat2024}
A. Plakhov. {\it Local structure of convex surfaces}. Sbornik Math {\bf 215}, 401-437 (2024).
%DOI: https://doi.org/10.4213/sm9921e  Russian version: pages  119-158 (2024).

  %\bibitem{ineq} A. Plakhov and V. Protasov.  {\it Local minima in Newton’s aerodynamic problem and inequalities between norms of partial derivatives}. J. Math. Anal. Appl. {\bf 543}, 128942 (2025).
%https://doi.org/10.1016/j.jmaa.2024.128942

%\bibitem{invisibility} A Plakhov and V Roshchina. {\it Invisibility in billiards}. Nonlinearity {\bf 24}, 847-854 (2011).

%\bibitem{3dir} A Plakhov and V Roshchina. {\it Fractal bodies invisible in 2 and 3 directions}. Discr. Contin. Dynam. Syst.-A {\bf 33}, 1615-1631 (2013). %doi:10.3934/dcds.2013.33.1615

  %\bibitem{invis2points} A Plakhov and V Roshchina. {\it Bodies with mirror surface invisible from two points.} Nonlinearity {\bf 27}, 1193-1203 (2014).
% doi:10.1088/0951-7715/27/6/1193

  %\bibitem{InftyDir}   A. Plakhov and V. Roshchina. {\it Invisibility in billiards is impossible in an infinite number of directions}. J. Dynam. Control Syst. {\bf 25}, 671-679, 2019.
  %http://link.springer.com/article/10.1007/s10883-019-09443-8

\bibitem{camo25}
A. Plakhov and V. Roshchina. {\it The problem of optimal camouflaging}. SIAM J. Math. Anal. {\bf 57}, 95-117 (2025). %DOI. 10.1137/24M1636976

\bibitem{Magnus}
A. Plakhov and T. Tchemisova. {\it Force acting on a spinning rough disk in a flow of non-interacting particles}.
Doklady Math. {\bf 79}, 132-135 (2009).

\bibitem{JDE}
A. Plakhov and T. Tchemisova. {\it Problems of optimal transportation on the circle and their mechanical applications}. J. Diff. Eqs. {\bf 262}, 2449-2492 (2017). %http://dx.doi.org/10.1016/j.jde.2016.10.049

\bibitem{PTG}
A. Plakhov, T. Tchemisova and P. Gouveia. {\it Spinning rough disk moving in a rarefied medium}. Proc. R. Soc. A. {\bf 466}, 2033-2055 (2010).

\bibitem{RenardyRogers} M. Renardy and R.C. Rogers. {\it An introduction to partial differential equations}. Texts in Applied Mathematics 13, Second ed. New York: Springer-Verlag ((2004).
  %\bibitem{temp} A. Plakhov and D. Torres.\, {\it Newton's aerodynamic problem in media of chaotically moving particles}. Sbornik: Math. {\bf 196}, 885-933 (2005).

 %\bibitem{Nonlinearity} A. Plakhov and P. Gouveia. {\it Problems of maximal mean resistance on the plane}. Nonlinearity {\bf 20}, 2271-2287 (2007).

      %\bibitem{Pogorelov}A.\,V. Pogorelov. {\it Extrinsic geometry of convex surfaces.} Providence, R.I.: American Mathematical Society (AMS). (1973).

 %\bibitem{Sto} L. Stojanov. {\it Note on the periodic points of the billiard}. J. Diff. Geom. {\bf 34}, 835–837 (1991).

 \bibitem{S}
 R. Schneider. {\it Convex bodies: the Brunn–Minkowski theory}. Encyclopedia Math. Appl., {\bf 44}, Cambridge Univ. Press, Cambridge, 1993, xiv+490 pp.

\bibitem{Tih}
 V.\,M. Tikhomirov. {\it Stories about maxima and minima}. Mathematical World, 1. AMS, Providence, RI, 1990.
 %[Translated from:  V.\,M. Tikhomirov. {\it Rasskazy o maksimumakh i minimumakh}. (Russian) ``Nauka'', Moscow, 1986.]

\bibitem{W}
G. Wachsmuth.\, {\it The numerical solution of Newton’s problem of least resistance}.\, Math. Program. A {\bf 147}, 331-350 (2014).
% DOI 10.1007/s10107-014-0756-2


\end{thebibliography}
\end{document}